\documentclass[10pt]{amsart}
\usepackage{amsfonts}
\usepackage{mathrsfs}
\usepackage{amscd}
\usepackage{amsmath}
\usepackage{amssymb}
\usepackage{latexsym}
\usepackage{lscape}
\usepackage{xypic}
\usepackage{comment}
\usepackage{amscd}
\usepackage{wasysym}  
\usepackage{tikz} 
\usetikzlibrary{matrix,arrows}
\usepackage{appendix}
\usepackage{geometry}
\usepackage[nice]{nicefrac}
\usepackage{dsfont}

\usepackage[colorlinks=true,pagebackref,hyperindex]{hyperref}

\usepackage{color}

\newcommand{\cb}{\color{blue}}

\newtheorem{thm}{Theorem}[section]

\newtheorem{prop}[thm]{Proposition}
\newtheorem{lem}[thm]{Lemma}

\newtheorem{lem-def}[thm]{Lemma-Definition}
\newtheorem{cor}[thm]{Corollary}

\theoremstyle{definition}

\newtheorem{rmk}[thm]{Remark}
\newtheorem{dfn}[thm]{Definition}

\numberwithin{equation}{section}

\newcommand{\nc}{\newcommand}

\nc{\on}{\operatorname}
\nc{\fraka}{{\mathfrak a}} \nc{\bba}{{\mathbf a}}
\nc{\frakb}{{\mathfrak b}}
\nc{\frakc}{{\mathfrak c}}
\nc{\frakd}{{\mathfrak d}}
\nc{\frake}{{\mathfrak e}}
\nc{\frakf}{{\mathfrak f}}
\nc{\frakg}{{\mathfrak g}}
\nc{\frakh}{{\mathfrak h}}
\nc{\fraki}{{\mathfrak i}}
\nc{\frakj}{{\mathfrak j}}
\nc{\frakk}{{\mathfrak k}}
\nc{\frakl}{{\mathfrak l}}
\nc{\frakm}{{\mathfrak m}}
\nc{\frakn}{{\mathfrak n}}
\nc{\frako}{{\mathfrak o}}
\nc{\frakp}{{\mathfrak p}}
\nc{\frakq}{{\mathfrak q}}
\nc{\frakr}{{\mathfrak r}}
\nc{\fraks}{{\mathfrak s}}
\nc{\frakt}{{\mathfrak t}}
\nc{\fraku}{{\mathfrak u}}
\nc{\frakv}{{\mathfrak v}}
\nc{\frakw}{{\mathfrak w}}
\nc{\frakx}{{\mathfrak x}}
\nc{\fraky}{{\mathfrak y}}
\nc{\frakz}{{\mathfrak z}}
\nc{\frakA}{{\mathfrak A}}
\nc{\frakB}{{\mathfrak B}}
\nc{\frakC}{{\mathfrak C}}
\nc{\frakD}{{\mathfrak D}}
\nc{\frakE}{{\mathfrak E}}
\nc{\frakF}{{\mathfrak F}}
\nc{\frakG}{{\mathfrak G}}
\nc{\frakH}{{\mathfrak H}}
\nc{\frakI}{{\mathfrak I}}
\nc{\frakJ}{{\mathfrak J}}
\nc{\frakK}{{\mathfrak K}}
\nc{\frakL}{{\mathfrak L}}
\nc{\frakM}{{\mathfrak M}}
\nc{\frakN}{{\mathfrak N}}
\nc{\frakO}{{\mathfrak O}}
\nc{\frakP}{{\mathfrak P}}
\nc{\frakQ}{{\mathfrak Q}}
\nc{\frakR}{{\mathfrak R}}
\nc{\frakS}{{\mathfrak S}}
\nc{\frakT}{{\mathfrak T}}
\nc{\frakU}{{\mathfrak U}}
\nc{\frakV}{{\mathfrak V}}
\nc{\frakW}{{\mathfrak W}}
\nc{\frakX}{{\mathfrak X}}
\nc{\frakY}{{\mathfrak Y}}
\nc{\frakZ}{{\mathfrak Z}}
\nc{\bbA}{{\mathbb A}}
\nc{\bbB}{{\mathbb B}}
\nc{\bbC}{{\mathbb C}}
\nc{\bbD}{{\mathbb D}}
\nc{\bbE}{{\mathbb E}}
\nc{\bbF}{{\mathbb F}} \nc{\bbf}{{\mathbf f}}
\nc{\bbG}{{\mathbb G}}
\nc{\bbH}{{\mathbb H}}
\nc{\bbI}{{\mathbb I}}
\nc{\bbJ}{{\mathbb J}}
\nc{\bbK}{{\mathbb K}}
\nc{\bbL}{{\mathbb L}}
\nc{\bbM}{{\mathbb M}}
\nc{\bbN}{{\mathbb N}}
\nc{\bbO}{{\mathbb O}}
\nc{\bbP}{{\mathbb P}}
\nc{\bbQ}{{\mathbb Q}}
\nc{\bbR}{{\mathbb R}}
\nc{\bbS}{{\mathbb S}}
\nc{\bbT}{{\mathbb T}}
\nc{\bbU}{{\mathbb U}}
\nc{\bbV}{{\mathbb V}}
\nc{\bbW}{{\mathbb W}}
\nc{\bbX}{{\mathbb X}}
\nc{\bbY}{{\mathbb Y}}
\nc{\bbZ}{{\mathbb Z}}
\nc{\calA}{{\mathcal A}}
\nc{\calB}{{\mathcal B}}
\nc{\calC}{{\mathcal C}}
\nc{\calD}{{\mathcal D}}
\nc{\calE}{{\mathcal E}}
\nc{\calF}{{\mathcal F}}
\nc{\calG}{{\mathcal G}}
\nc{\calH}{{\mathcal H}}
\nc{\calI}{{\mathcal I}}
\nc{\calJ}{{\mathcal J}}
\nc{\calK}{{\mathcal K}}
\nc{\calL}{{\mathcal L}}
\nc{\calM}{{\mathcal M}}
\nc{\calN}{{\mathcal N}}
\nc{\calO}{{\mathcal O}}
\nc{\calP}{{\mathcal P}}
\nc{\calQ}{{\mathcal Q}}
\nc{\calR}{{\mathcal R}}
\nc{\calS}{{\mathcal S}}
\nc{\calT}{{\mathcal T}}
\nc{\calU}{{\mathcal U}}
\nc{\calV}{{\mathcal V}}
\nc{\calW}{{\mathcal W}}
\nc{\calX}{{\mathcal X}}
\nc{\calY}{{\mathcal Y}}
\nc{\calZ}{{\mathcal Z}}

\nc{\scrA}{{\mathscr A}}
\nc{\scrB}{{\mathscr B}}
\nc{\scrR}{{\mathscr R}}

\nc{\Bmu}{\mbox{$\raisebox{-0.59ex}{$l$}\hspace{-0.18em}\mu\hspace{-0.88em}\raisebox{-0.98ex}{\scalebox{2}{$\color{white}.$}}\hspace{-0.416em}\raisebox{+0.88ex}{$\color{white}.$}\hspace{0.46em}$}{}}

\nc{\bnu}{{\bar{ \nu}}}

\nc{\olO}{\bar{\calO}}

\nc{\al}{{\alpha}} 
\nc{\be}{{\beta}}
\nc{\ga}{{\gamma}} \nc{\Ga}{{\Gamma}}
 \nc{\hGa}{\hat{\Gamma}}
\nc{\ve}{{\varepsilon}} 
\nc{\la}{{\lambda}} \nc{\La}{{\Lambda}}
\nc{\om}{\omega} \nc{\Om}{\Omega} 
\nc{\sig}{{\sigma}} \nc{\Sig}{{\Sigma}}

\nc{\tnb}{\psi_{\rm tame}}
\nc{\oM}{\overline{{M}}}
\nc{\op}{{\on{op}}}
\nc{\ad}{{\on{ad}}}
\nc{\alg}{{\on{alg}}}
\nc{\Ad}{{\on{Ad}}}
\nc{\Adm}{{\on{Adm}}} \nc{\aff}{{\on{af}}}
\nc{\Aut}{{\on{Aut}}}
\nc{\Bun}{{\on{Bun}}}
\nc{\cha}{{\on{char}}}
\nc{\der}{{\on{der}}}
\nc{\Der}{{\on{Der}}}
\nc{\diag}{{\on{diag}}}
\nc{\End}{{\on{End}}}
\nc{\Fl}{{\calF\!\ell}}
\nc{\Tr}{{\on{Transp}}}
\nc{\TR}{{\calT\!\calR}}
\nc{\Gal}{{\on{Gal}}}
\nc{\Gr}{{\on{Gr}}}
\nc{\rH}{{\on{H}}}
\nc{\Hom}{{\on{Hom}}}
\nc{\IC}{{\on{IC}}}
\nc{\id}{{\on{id}}}
\nc{\Id}{{\on{Id}}}
\nc{\ind}{{\on{ind}}}
\nc{\Ind}{{\on{Ind}}}
\nc{\Lie}{{\on{Lie}}}
\nc{\Pic}{{\on{Pic}}}
\nc{\pr}{{\on{pr}}}
\nc{\Res}{{\on{Res}}}
\nc{\res}{{\on{res}}} \nc{\Sat}{{\on{Sat}}}
\nc{\s}{{\on{sc}}}
\nc{\drv}{{\on{der}}}
\nc{\sgn}{{\on{sgn}}}
\nc{\Spec}{{\on{Spec}}}\nc{\Spf}{\on{Spf}} 
\nc{\Sph}{\on{Sph}}
\nc{\St}{{\on{St}}}
\nc{\tr}{{\on{tr}}}
\nc{\Mod}{{\mathrm{-Mod}}}
\nc{\Hilb}{{\on{Hilb}}} 
\nc{\Ext}{{\on{Ext}}} 
\nc{\vs}{{\on{Vec}}}
\nc{\ev}{{\on{ev}}}
\nc{\nO}{{\breve{\calO}}}
\nc{\tS}{{\tilde{S}}}
\nc{\spe}{{\on{sp}}}
\nc{\loc}{{\on{loc}}}

\nc{\nscrR}{{\mathscr{R}^{\on{nr}}}}

\nc{\GL}{{\on{GL}}}
\nc{\U}{{\on{U}}}
\nc{\Gl}{\on{Gl}} 
\nc{\GSp}{{\on{GSp}}}
\nc{\gl}{{\frakg\frakl}}
\nc{\SL}{{\on{SL}}} 
\nc{\SU}{{\on{SU}}} 
\nc{\SO}{{\on{SO}}}
\nc{\PGL}{{\on{PGL}}}

\nc{\Conv}{{\on{Conv}}}
\nc{\Rep}{{\on{Rep}}}
\nc{\Dom}{{\on{Dom}}}
\nc{\red}{{\on{red}}}
\nc{\act}{{\on{act}}}
\nc{\nr}{{\on{nr}}}
\nc{\ctf}{{\on{ctf}}}

\nc{\str}{{\on{-}}} 
\nc{\os}{{\bar{s}}}
\nc{\oeta}{{\bar{\eta}}}

\nc{\hookto}{\hookrightarrow}
\nc{\longto}{\longrightarrow}
\nc{\leftto}{\leftarrow}
\nc{\onto}{\twoheadrightarrow}
\nc{\lonto}{\twoheadleftarrow}

\nc{\uG}{{\underline{G}}}
\nc{\uA}{{\underline{A}}}
\nc{\uS}{{\underline{S}}}
\nc{\uT}{{\underline{T}}}
\nc{\uM}{{\underline{M}}}
\nc{\uP}{{\underline{P}}}
\nc{\uB}{{\underline{B}}}
\nc{\uN}{{\underline{N}}}

\nc{\ucG}{{\underline{\calG}}}
\nc{\ucA}{{\underline{\calA}}}
\nc{\ucS}{{\underline{\calS}}}
\nc{\ucT}{{\underline{\calT}}}
\nc{\ucM}{{\underline{\calM}}}
\nc{\ucP}{{\underline{\calP}}}
\nc{\ucN}{{\underline{\calN}}}
\nc{\ucH}{{\underline{\calH}}}
\nc{\uH}{{\underline{H}}}

\nc{\AffSch}{{\on{AffSch}}}

\nc{\bF}{{\breve{F}}}

\nc{\oFl}{{\overline{\Fl}}} 
\nc{\bU}{{\overline{U}}}
\nc{\tGr}{{\tilde{\Gr}}}
\nc{\cGr}{\calG\! r}
\nc{\oGr}{\overline{\on{Gr}}} 
\nc{\ocGr}{\overline{\calG\! r}}
\nc{\co}{{\colon}}
\nc{\sch}[1]{(Sch/{#1})}
\nc{\HypLoc}[1]{HypLoc({#1})}

\nc{\ohtimes}{\stackrel{!}{\otimes}}
\nc{\boxtilde}{\widetilde{\boxtimes}}
\nc{\vstar}{{\varhexstar}}

\nc{\Div}{\on{Div}}

\nc{\ucI}{{\underline{\calI}}}

\nc{\bslash}{\backslash}
\nc{\algQl}{{\bar{\bbQ}_\ell}}
\nc{\sF}{{\bar{F}}}
\nc{\nF}{{\breve{F}}}
\nc{\nW}{{W^{\on{nr}}}}
\nc{\sk}{{\bar{k}}}
\nc{\cont}{\on{c}}
\nc{\Supp}{\on{Supp}}
\nc{\blt}{\bullet}  
\nc{\dom}{\on{dom}}
\nc{\scon}{{\on{sc}}} 
\nc{\Affine}{\on{Aff}} 
\nc{\nscrA}{\mathscr{A}^{\on{nr}}} 
\nc{\nfraka}{{\bbf^{\on{nr}}}}
\nc{\ran}{{\rangle}}
\nc{\lan}{{\langle}}
\nc{\bk}{{\bar{k}}}
\nc{\tF}{{\tilde{F}}}
\nc{\sS}{{\bar{S}}}
\nc{\LG}{{^\text{L}\hspace{-0.04cm}G}}
\nc{\LL}{{^\text{L}\hspace{-0.07cm}L}}

\nc{\pot}[1]{ [\hspace{-0,5mm}[ {#1} ]\hspace{-0,5mm}] }
\nc{\rpot}[1]{ (\hspace{-0,7mm}( {#1} )\hspace{-0,7mm}) }

\nc{\defined}{\hspace{0.1cm}\stackrel{\text{\tiny \rm def}}{=}\hspace{0.1cm}}

\begin{document}

\title[Normality and Cohen-Macaulayness]{Normality and Cohen-Macaulayness \\ of parahoric local models}
\author[T.\,J.\,Haines and T.\,Richarz]{by Thomas J. Haines and Timo Richarz}

\address{Department of Mathematics, University of Maryland, College Park, MD 20742-4015, DC, USA}
\email{tjh@math.umd.edu}

\address{Fachbereich Mathematik, TU Darmstadt, Schlossgartenstrasse 7, 64289 Darmstadt, Germany}
\email{richarz@mathematik.uni-darmstadt.de}

\thanks{Research of T.R.~partially funded by the Deutsche Forschungsgemeinschaft (DFG, German Research Foundation) - 394587809.}

\maketitle

\begin{abstract} 
We study the singularities of integral models of Shimura varieties and moduli stacks of shtukas with parahoric level structure.
More generally our results apply to the Pappas-Zhu and Levin mixed characteristic parahoric local models, and to their equal characteristic analogues. 
For any such local model we prove under minimal assumptions that the entire local model is normal with reduced special fiber and, if $p > 2$, it is also Cohen-Macaulay. This proves a conjecture of Pappas and Zhu, and shows that the integral models of Shimura varieties constructed by Kisin and Pappas are Cohen-Macaulay as well. 



\end{abstract}

\tableofcontents


\thispagestyle{empty}

\section{Introduction}
Parahoric local models are flat projective schemes over a discrete valuation ring which are designed to model the \'etale local structure in the bad reduction of Shimura varieties \cite{KP18} or moduli stacks of shtukas \cite{AH19} with parahoric level structure. Thus, local models provide a tool to study the singularities appearing in the reduction of these spaces. The subject dates back to early 1990's, starting with the work Deligne-Pappas \cite{DP94}, Chai-Norman \cite{CN90, CN92} and de Jong \cite{dJ93}, and was formalized in the book of Rapoport-Zink \cite{RZ96}. For further details and references, we refer to the survey of Pappas, Rapoport and Smithling \cite{PRS13}.

The simplest example is the case of the classical modular curve $X_0(p)$ with $\Ga_0(p)$-level structure. In this case, the local model is a $\bbP^1_{\bbZ_p}$ blown up in the origin of the special fiber $\bbP^1_{\bbF_p}$. This models the reduction modulo $p$ of $X_0(p)$, which is visualized as the famous picture of two irreducible components crossing transversally at the supersingular points.

In a breakthrough, a group theoretic construction for parahoric local models in mixed characteristic was given by Pappas-Zhu \cite{PZ13}: fix a local model triple $(G,\{\mu\},K)$ where $G$ is a reductive group over $\bbQ_p$, $K\subset G(\bbQ_p)$ is a parahoric subgroup and $\{\mu\}$ is a (not necessarily minuscule) geometric conjugacy class of one-parameter subgroups in $G$. Under a tamely ramified hypothesis on $G$ which was further relaxed by Levin \cite{Lev16}, Pappas and Zhu construct a flat projective $\bbZ_p$-scheme $M_{\{\mu\}}=M_{(G,\{\mu\},K)}$ called the (parahoric) local model. Assume now that $p$ does not divide the order of $\pi_1(G_\der)$, e.g., the case $G=\PGL_n$ and $p\!\mid\!n$ is excluded. Then it is shown in \cite{PZ13} (cf.~also \cite{Lev16}) that $M_{\{\mu\}}$ is normal with reduced special fiber. Further, they conjecture that $M_{\{\mu\}}$ is Cohen-Macaulay, cf.~\cite[Rmk.~9.5 (b)]{PZ13}. This conjecture was proven in \cite[Cor.~9.4]{PZ13} when $K$ is very special (in which case the special fiber of $M_{\{\mu\}}$ is irreducible), and by He \cite{He13} when the group $G$ is unramified, $K$ is an Iwahori subgroup and $\{\mu\}$ is minuscule. Also some explicit special cases were treated earlier by G\"ortz (unpublished) following the method outlined in \cite[4.5.1]{Go01}. 
 
The purpose of this article is to show, for a general local model triple $(G,\{\mu\},K)$ as above under minimal assumptions (cf.~Remark \ref{optimal_rmk} below), that the local model $M_{\{\mu\}}$ itself is normal with reduced special fiber, and, if $p>2$, also that $M_{\{\mu\}}$ is Cohen-Macaulay, cf.~Theorem \ref{special_fiber_thm} below. 
This recovers and extends the results on the geometry of local models obtained in \cite{PZ13} and \cite{Lev16}.
In particular, under these assumptions we obtain the Cohen-Macaulayness of $M_{\{\mu\}}$ for all not necessarily minuscule $\{\mu\}$; this answers a question of He, and seems to be new even in the case of $G=\GL_2$ (where the assumption $p > 2$ alone is sufficient). 
Theorem \ref{special_fiber_thm} also seems to be the first Cohen-Macaulayness result which applies to ramified groups and Iwahori level structure.

Our proof of normality follows the original argument of \cite{PZ13} relying on the Coherence Conjecture proven by Zhu \cite{Zhu14}. The assumption on the normality of Schubert varieties inside the $\mu$-admissible locus is satisfied whenever $p\nmid |\pi_1(G_\der)|$ by Pappas-Rapoport \cite{PR08}, but also in other cases of interest, e.g., whenever $K$ is an Iwahori subgroup and $\bar{\mu} \in X_*(T)_I$ is minuscule for the \'echelonnage root system (e.g. $G$ is unramified and $\mu$ is minuscule for the absolute roots), cf.~\cite{HLR}. The proof of the Cohen-Macaulayness of $M_{\{\mu\}}$ is very different from the approaches in \cite[4.5.1]{Go01} and \cite{He13}. Here we reduce the problem first to equal characteristic where we can use the result \cite[Thm.~6.5]{Zhu14} (which assumes $p>2$) to apply the powerful method of Frobenius splittings on the whole local model. This reduces us to the situation of a scheme over a finite field whose complement along some divisor $D$ is Cohen-Macaulay, and which is Frobenius split relative to $D$.
Surprisingly (at least for the authors), the Cohen-Macaulayness at points lying on $D$ then follows from a very nice homological algebra lemma which we found in the work of Blickle-Schwede \cite[Ex.~5.4]{BS13} (cf.~Proposition \ref{CM_split} below).

\section{Main result}\label{main_result_sec}
Let $p\in\bbZ_{>0}$ be a prime number. Let $F$ be a non-archimedean local field with ring of integers $\calO_F$ and finite residue field $k=k_F$ of characteristic $p$, i.e. either $F/\bbQ_p$ is a finite extension or $F\simeq k\rpot{t}$ is a local function field. We fix a separable closure $\sF/F$.

We fix a triple $(G,\{\mu\},\calG_\bbf)$ where $G$ is a connected reductive $F$-group, $\{\mu\}$ a (not necessarily minuscule) conjugacy class of geometric cocharacters defined over a finite separable extension $E/F$, and $\calG_\bbf$ is a parahoric $\calO_F$-group scheme associated with some facet $\bbf\subset \scrB(G,F)$.

We put the following assumption on $G$. If $F/\bbQ_p$, we assume that $G\simeq \Res_{K/F}(G_1)$ where $K/F$ is a finite extension, and $G_1$ is a connected reductive $K$-group which splits over a tamely ramified extension. If $F\simeq k\rpot{t}$, we assume that in the simply connected cover of the derived group $G_\scon\simeq \prod_{j\in J}\Res_{F_j/F}(G_j)$ each absolutely almost simple factor $G_j$ splits over a tamely ramified extension of $F_j$. Here $J$ is a finite index set, and each $F_j/F$ is a finite separable field extension.

Attached to the triple $(G,\{\mu\},\calG_\bbf)$ is the flat projective $\calO_E$-scheme
\[
M_{\{\mu\}}=M_{(G,\{\mu\},\calG_\bbf)},
\]
called the ({\em flat}) \emph{local model}. It is defined in \cite{PZ13, Lev16} (cf.~also \cite[\S 4]{HRb}) if $F/\bbQ_p$, and in \cite{Zhu14, Ri16} if $F\simeq k\rpot{t}$. The generic fiber is the Schubert variety $\Gr_G^{\leq \{\mu\}}\to \Spec(E)$ in the affine Grassmannian of $G\otimes_F E$ associated with the class $\{\mu\}$. The special fiber is equidimensional, but not irreducible in general, and is equipped with a closed embedding
\begin{equation}
\overline{M}_{\{\mu\}}:=M_{\{\mu\}}\otimes_{\calO_E}k_E \;\hookto\;\Fl_{G^\flat,\bbf^\flat}\otimes_kk_E,
\end{equation}
where $k_E$ is the residue field of $E$. Further, if $F\simeq k\rpot{t}$, we have $(G^\flat,\bbf^\flat)=(G,\bbf)$. If $F/\bbQ_p$, the pair $(G^\flat,\bbf^\flat)$ is an equal characteristic analogue over a local function field $k\rpot{u}$ of the pair $(G_1,\bbf_1)$ where $\bbf_1$ is the facet corresponding to $\bbf$ under $\scrB(G,F)\simeq \scrB(G_1,K)$, cf.~\S\ref{recollections} for details. By \cite[Thm.~5.14]{HRb}, with no restriction on $p$, we have for the reduced geometric special fiber 
\[
\big(\overline{M}_{\{\mu\}}\otimes_{k_E}\bar{k}\big)_\red\;=\;\calA(G,\{\mu\}),
\]
where $\calA(G,\{\mu\})\subset \Fl_{G^\flat,\bbf^\flat}\otimes_k\bar{k}$ denotes the union of the $(\bbf^\flat,\bbf^\flat)$-Schubert varieties indexed by the $\{\mu\}$-admissible set of Kottwitz-Rapoport.

The following theorem proves results on the geometry of local models under weaker hypotheses than the hypothesis $p\nmid |\pi_1(G_\der)|$ in \cite[Thm.~9.3]{PZ13} and \cite[Thm.~4.3.2]{Lev16}. We recover as very special cases the results of \cite{PZ13} in this direction and those of \cite{He13}, which treats unramified groups, Iwahori level, and minuscule $\{\mu\}$\footnote{But we note that \cite{He13} does not require the assumption $p>2$.}.

\begin{thm} \label{special_fiber_thm}
Let $(G,\{\mu\},\calG_\bbf)$ be as above, and assume that $\Gr_{G,\bar{F}}^{\leq \{\mu\}}$ is normal and that all {\em maximal} $(\bbf^\flat,\bbf^\flat)$-Schubert varieties inside $\calA(G,\{\mu\})$ are normal \textup{(}see Remark \ref{intro_rmk} ii\textup{)} below for situations when this is satisfied\textup{)}.
\begin{enumerate}
\item[i)] The special fiber $\overline{M}_{\{\mu\}}$ is geometrically reduced. More precisely, as closed subschemes of $\Fl_{G^\flat,\bbf^\flat}\otimes_k\bar{k}$, one has
\[
\overline{M}_{\{\mu\}}\otimes_{k_E}\bar{k}\;=\; \calA(G,\{\mu\}).
\]
Further, each irreducible component of $\overline{M}_{\{\mu\}}\otimes\bar{k}$ is normal, Cohen-Macaulay, with only rational singularities, and Frobenius split.
\item[ii)] The local model $M=M_{\{\mu\}}$ is normal and, if $p>2$, also Cohen-Macaulay with dualizing sheaf given by the double dual of the top differentials 
\[
\om_{M}\;=\;(\Om^{d}_{M/\calO_E})^{*,*},
\]
where $d=\dim(M_{E})$ is the dimension of the generic fiber.
\end{enumerate}
\end{thm}

\begin{rmk} \label{intro_rmk} i) In view of the results in \cite{KP18} (resp.~\cite[\S 3]{AH19}) the corresponding integral models of Shimura varieties (resp.~moduli stacks of shtukas) with parahoric level structure are normal and Cohen-Macaulay as well.\smallskip\\
ii) The Schubert variety $\Gr_{G,\bar{F}}^{\leq \{\mu\}}$ and all Schubert varieties inside $\calA(G,\{\mu\})$ are known to be normal
\begin{itemize}
\item if $p\nmid |\pi_1(G_\der)|$ and $G\simeq \Res_{K/F}(G_1)$ with $G_1$ tamely ramified as above (without any restrictions on $\{\mu\}$), cf.~\cite[Thm.~8.4]{PR08}; \cite[Thm.\,5.14]{HRb}; end of $\S\ref{Absolutely_Simple_Red}$); 
\item if $\bar{\mu} \in X_*(T)_I$ is minuscule for the \'echelonnage root system, and the facet $\bbf$ contains a special vertex, e.g.,\,$G$ is unramified, $\calG_\bbf(\calO_F)$ an Iwahori subgroup and $\mu$ is minuscule (without any restrictions on $p$), cf.~\cite{HLR}. In general, $\bar{\mu}$ being minuscule for the \'echelonnage roots implies that $\mu$ is minuscule for the absolute roots, by \cite[$\S4.3$]{Hai18} (but not conversely).
\end{itemize}
We also remark that when ${\rm char}(F) = 0$, $\Gr^{\leq \{ \mu\}}_{G, \bar{F}}$ is known to be normal. 
However, 
when $p \mid |\pi_1(G_\der)|$, there are non-normal Schubert varieties in $\Fl_{G^\flat, \bbf^\flat}$, even for $G^\flat = {\rm PGL}_2$ over $\bar{\mathbb F}_2$, cf.\,\cite{HLR}.
\end{rmk}

Thus, Theorem \ref{special_fiber_thm}\,i) gives new cases of normal local models with reduced special fiber. The proof follows the original argument of Pappas-Zhu, using as a key input the Coherence Conjecture proved by Zhu \cite{Zhu14} (cf.~also \cite[\S9.2.2, (9.19)]{PZ13}). The application of the Coherence Conjecture is justified by our assumption on the normality of Schubert varieties inside $\calA(G,\{\mu\})$.

For ii), the normality of $M_{\{\mu\}}$ is an immediate consequence of i) by \cite[Prop.~9.2]{PZ13}. As mentioned above, the Cohen-Macaulayness of $M_{\{\mu\}}$ can be deduced from Proposition \ref{CM_split} below combined with the well-known theorem of Zhu \cite[Thm.~6.5]{Zhu14} which is also the key to the Coherence Conjecture. In particular, our method avoids using any finer geometric structure of the admissible locus $\calA(G,\{\mu\})$ as for example in \cite[\S4.5.1]{Go01} or \cite{He13}.

\begin{cor} \label{Invariance_Cor}
Let $(G', \{\mu'\}, \calG_{\bbf'})\to (G, \{\mu\}, \calG_{\bbf})$ a map of local model triples as above which induces $G'_\ad\simeq G_\ad$ on adjoint groups. Then the induced map on local models
$$
M_{(G', \{\mu'\}, \calG_{\bbf'})} \, \rightarrow \, M_{(G, \{\mu\}, \calG_{\bbf})} \otimes_{\calO_{E}} \calO_{E'},
$$
is a finite, birational universal homeomorphism where $E'/F$ denotes the reflex field of $\{\mu'\}$, naturally an overfield of $E$. In particular, if $M_{(G, \{\mu\}, \calG_{\bbf})}$ is normal \textup{(}see Theorem \ref{special_fiber_thm}\textup{)}, this map is an isomorphism.
\end{cor}

\begin{rmk}\label{optimal_rmk}
The assumption that $M_{(G,\{\mu\}, \calG_{\bbf})}$ is normal is equivalent to the assumption that its generic fiber ${\rm Gr}^{\leq \{\mu\}}_{G, \bar{F}}$ is normal and all maximal $(\bbf^\flat,\bbf^\flat)$-Schubert varieties inside $\calA(G,\{\mu\})$ are normal. One implication is given by Theorem \ref{special_fiber_thm} ii). For the other implication, 
apply Corollary \ref{Invariance_Cor} to a $z$-extension $G'$ of $G$ with $G'_{\rm der} = G'_{\rm sc}$ and invoke Remark \ref{intro_rmk} ii) for $G'$. Thus the assumption on the normality of Schubert varieties appearing in the special fiber is {\it optimal} in the following sense: if one of the maximal $(\bbf^\flat,\bbf^\flat)$-Schubert varieties inside $\calA(G,\{\mu\})$ is not normal, then the whole local model $M_{(G, \{\mu\}, \calG_{\bbf})}$ is not normal.  
\end{rmk}

The proof of Corollary \ref{Invariance_Cor} is given in \S\ref{Invariance_Cor_Sec} below.
This relates to the modified local models $\mathbb M^{\rm loc}_{\calG}(G, \{\mu\})$ of \cite[\S2.6]{HPR}. For an appropriate choice of $z$-extension $(\tilde{G}, \tilde{\mu}) \to (G_\ad,\mu_\ad)$ with $\tilde{G}_{\rm der} = G_\scon$ and parahoric group $\tilde{\calG}$ corresponding to $\calG$, \cite[$\S2.6$]{HPR} defines 
$$
\mathbb M^{\rm loc}_{\calG}(G,\{\mu\}) \, \overset{\rm def}{=} \,  M_{(\tilde{G}, \tilde{\mu}, \tilde{\calG})} \otimes_{\calO_{E_\ad}} \calO_E.
$$
The {\em raison d'\^{e}tre} for ${\mathbb M}^{\rm loc}_{\calG}(G, \{\mu\})$ is that it should possess the geometric properties which can fail to hold for $M_{(G, \{\mu\}, \calG)}$ when $p \mid |\pi_1(G_{\rm der})|$. This is indeed the case.

\begin{cor} \label{modified_Cor}
The modified local model ${\mathbb M}^{\rm loc}_{\calG}(G, \{\mu\})$ is normal with reduced special fiber, and is also Cohen-Macaulay if $p> 2$.
\end{cor}
As $|\pi_1(\tilde{G}_{\rm der})| = 1$, Corollary \ref{modified_Cor} is immediate from Theorem \ref{special_fiber_thm} and Remark \ref{intro_rmk}. Only the Cohen-Macaulayness is new: the other statements were observed in \cite[Rmk.\,2.9]{HPR}. 
Corollaries \ref{Invariance_Cor} and \ref{modified_Cor} imply that the modified local model $\mathbb M^{\rm loc}_{\calG}(G, \{\mu\})$ always identifies with the normalization of $M_{(G, \{\mu\}, \calG)}$.
The Kisin-Pappas integral models $\mathcal S_K({\bf G}, {\bf X})$ of Shimura varieties in \cite{KP18} can be defined when $p > 2$ (even when $p \mid |\pi_1(G_{\rm der})|$), and are always ``locally modeled'' by associated modified local models $\mathbb M^{\rm loc}_{\calG}(G, \{\mu\})$ (cf.\,\cite[Thm.\,3.1]{HPR}).  
Therefore, Corollary \ref{modified_Cor} implies that all Kisin-Pappas integral models are Cohen-Macaulay.

\medskip
\noindent\textbf{Acknowledgements.} The starting point of this note was a question of Jo\~ao Louren\c{c}o, and we thank him for his interest in our work. Further, we thank Michael Rapoport and the University of Maryland for financial and logistical support. Also we heartily thank Karl Schwede for explaining \cite[Ex.~5.4]{BS13} to us, and Ulrich G\"ortz for interesting email exchanges.

\section{Preliminaries on Schubert varieties}\label{prelim_Schubert}
We introduce temporary notation for use within \S\ref{prelim_Schubert}.  Let $k$ be an algebraically closed field, and let $F=k\rpot{t}$ denote the Laurent series field. Let $G$ be a connected reductive $F$-group. Let $\bbf, \bbf'\subset \scrB(G,F)$ be facets of the Bruhat-Tits building. 
Let $\calG_\bbf$ (resp. $\calG_{\bbf'}$) be the associated parahoric $\calO_F$-group scheme. The loop group $LG$ (resp.~$L^+\calG_\bbf$) is the functor on the category of $k$-algebras $R$ defined by $LG(R)=G(R\rpot{t})$ (resp.~$L^+\calG_\bbf(R)=\calG_\bbf(R\pot{t})$). Then $L^+\calG_\bbf\subset LG$ is a subgroup functor, and the {\it twisted affine flag variety} is the \'etale quotient
\[
\Fl_{G, \bbf}\defined LG/L^+\calG_\bbf,
\]
which is representable by an ind-projective $k$-ind-scheme. Associated with an element $w\in LG(k)$, we define the $(\bbf',\bbf)$-Schubert variety $S_w=S_w(\bbf',\bbf)\subset \Fl_{G, \bbf}$ as the reduced $L^+\calG_{\bbf'}$-orbit closure of $w\cdot e$ where $e\in \Fl_{G, \bbf}$ denotes the base point.

Let $G_\scon\to G_\der\subset G$ denote the simply connected cover of the derived group. Then $G_\scon\simeq \prod_{j\in J}\Res_{F_j/F}(G_j)$ where $J$ is a finite index set, each $F_j/F$ is a finite separable field extension, and each $G_j$ is an absolutely almost simple, simply connected, reductive $F_j$-group. Under the induced map on buildings $\scrB(G_\scon,F)\to \scrB(G,F)$ the facets correspond bijectively to each other, and we denote by $\bbf_j\subset \scrB(G_j,F_j)=\scrB(\Res_{F_j/F}(G_j),F)$ (cf.~\cite[Prop.~4.6]{HRb}) the factor corresponding to $j\in J$ of the facet $\bbf\subset \scrB(G,F)$. 

\begin{prop} \label{normal_prop}
Let $S_w=S_w(\bbf',\bbf)$ be any Schubert variety. 
\begin{enumerate}
\item[i)] The normalization $\tilde{S}_w\to S_w$ is a finite birational universal homeomorphism. Further, $\tilde{S}_w$ is also Cohen-Macaulay, with only rational singularities, and Frobenius split if $\on{char}(k)>0$.
\item[ii)] For each $j\in J$, there exists an alcove $\bba_j\subset \scrB(G_j,F_j)$ containing $\bbf_j$ in its closure and an element $w_j\in LG_j(k)$ together with an isomorphism of $k$-schemes
\[
\tilde{S}_w\;\simeq\; \prod_{j\in J}\tilde{S}_{w_j},
\]  
where  each $\tilde S_{w_j}$ is the normalization of a $(\bba_j,\bbf_j)$-Schubert variety $S_{w_j}=S_{w_j}(\bba_j,\bbf_j)\subset \Fl_{G_j,\bbf_j}$. Further, if $\tilde S_w=S_w$ is normal, then each $\tilde S_{w_j}=S_{w_j}$ is normal as well.
\end{enumerate}
\end{prop}

We first reduce the proof of Proposition \ref{normal_prop} to the case where $\bbf'=\bba$ is an alcove, and where $S_w=S_w(\bba,\bbf)$ is contained in the neutral component $\Fl_{G,\bbf}^o\subset \Fl_{G,\bbf}$, cf.~\S\ref{Alcove_Red}. The proof of part i) is given in \S\ref{proof_i}, and of part ii) in \S\ref{proof_ii} below.

\subsection{Alcove reduction}\label{Alcove_Red} 
By \cite[Thm 7.4.18 (i)]{BT72}, there exists a maximal $F$-split torus $S\subset G$ such that $\bbf,\bbf'$ are both contained in the apartment $\scrA(G,S,F)$. The centralizer $T=C_G(S)$ is a maximal torus (because $G$ is quasi-split over $F$ by Steinberg's theorem), and we denote by $N=N_G(S)$ the normalizer. Let $\tilde{W}=N(F)/\calT(\calO_F)$ be the Iwahori-Weyl group (or extended affine Weyl group) where $\calT$ denotes the neutral component of the lft N\'eron model of $T$. We also denote $W_\bbf=(L^+\calG_{\bbf}(k)\cap N(F))/\calT(\calO_F)$, and likewise $W_{\bbf'}$.

\subsubsection{Step 1} \label{Step 1} As $\tilde{W}$ acts transitively on the set of alcoves in $\scrA(G,S,F)$, there exists an $n\in N(F)$ such that $n\cdot \bbf'$ and $\bbf$ are contained in the closure of the same alcove. Further, left multiplication by $n$ on $\Fl_{G,\bbf}$ induces an isomorphism
\[
S_w(\bbf',\bbf)\,\simeq\, S_{n\cdot w}(n\cdot \bbf',\bbf),
\]
where we use that $S_{n\cdot w}(n\cdot \bbf',\bbf)$ denotes the reduced orbit closure of $n\cdot w\in LG(k)$ under the group $L^+\calG_{n\cdot \bbf'}=nL^+\calG_{ \bbf'}n^{-1}$. Hence, we may assume the facets $\bbf, \bbf'$ are in the closure of a single alcove $\bba\subset \scrA(G,S,F)$.

\subsubsection{Step 2} By \cite[Prop 8]{HR08}, there exists an element $w_0\in \tilde{W}$ such that
\[
L^+\calG_{\bbf'}(k)\cdot w\cdot L^+\calG_\bbf(k) = L^+\calG_{\bbf'}(k)\cdot \dot{w}_0\cdot L^+\calG_{\bbf}(k),
\]
where $\dot{w}_0\in N(F)$ denotes any representative of $w_0$. This implies the equality $S_w(\bbf',\bbf)=S_{\dot{w}_0}(\bbf',\bbf)$ of Schubert varieties. Hence, we reduce to the case where $w=\dot{w}_0\in N(F)$.

\subsubsection{Step 3} Corresponding to the choice of $\bba$, we have the decomposition $\tilde{W}=W_\aff\rtimes \Om$ where $W_\aff=W_\aff(\breve{\Sig})$ is the affine Weyl group for the \'echelonnage root system $\breve{\Sig}$, and $\Om=\Om_\bba$ denotes the stabilizer of $\bba$ in $\tilde{W}$, cf.~\cite[Prop 12 ff.]{HR08}. Thus $\tilde{W}$ has the structure of a quasi-Coxeter group. Let $w_0\in W_{\bbf'}\cdot w\cdot W_{\bbf}$ be the unique left $\bbf'$-maximal element among all right $\bbf$-minimal elements, cf.~\cite[Lem 1.6 ii)]{Ri13}. Then for any representative $\dot{w}_0\in N(F)$ of $w_0$, we have $S_{w}(\bbf',\bbf)=S_{\dot{w}_0}(\bba,\bbf)$, cf.~\cite[Proof of Prop 2.8]{Ri13}. Hence, we reduce further to the case where $\bbf'=\bba$.

\subsubsection{Step 4} The class of $w$ in $\tilde{W}$ can be written uniquely in the form $\tau \cdot w_\aff$ with $w_\aff\in W_\aff$, $\tau\in \Om$. Left multiplication by any lift $\dot{\tau}\in N(F)$ induces an isomorphism of Schubert varieties $S_{\dot{w}_\aff}(\bba,\bbf)= S_w(\bba,\bbf)$. Hence, we reduce to the case where $\tau=1$, i.e., $w=w_\aff\in W_\aff$. This means that $S_{w}(\bba,\bbf)$ is contained in the neutral component $\Fl_{G,\bbf}^o\subset \Fl_{G,\bbf}$ (see \cite[$\S5$]{PR08}).

\begin{cor}\label{Reduction_Cor}
If the statement of Proposition \ref{normal_prop} holds for the Schubert varieties $S_w(\bba,\bbf)$ for all maximal split $F$-tori $S\subset G$, one choice of alcove $\bba\subset \scrA(G,S,F)$ containing $\bbf$ in its closure and all $w\in W_{\aff}$, then Proposition \ref{normal_prop} holds for the Schubert varieties $S_w(\bbf',\bbf)$ for all facets $\bbf,\bbf'\subset \scrB(G,F)$ and all $w\in LG(k)$.
\end{cor}
\begin{proof}
The corollary is immediate from Steps 1-4 above. Note that it is enough to show Proposition \ref{normal_prop} ii) for a single alcove $\bbf' = \bba\subset \scrA(G,S,F)$ because $W_\aff$ acts (simply) transitively on these alcoves.
\end{proof}

\subsection{A lemma on orbits}\label{orbit_lem} The following lemma helps to control orbits in partial affine flag varieties. 

\begin{lem}\label{right_min_lem} 
Let $w\in \tilde{W}$ be right $\bbf$-minimal. Then the natural map
\begin{equation}\label{Iso_Cells}
L^+\calG_\bba\cdot w\cdot L^+\calG_\bba/L^+\calG_\bba\,\longto\, L^+\calG_\bba \cdot w\cdot L^+\calG_\bbf/L^+\calG_\bbf
\end{equation}
is an isomorphism.
\end{lem}
\begin{proof} In the split case, this was proved in \cite[Lem.\,4.2]{HRa} using negative loop groups. In the general case the required properties of negative loop groups are not yet available, so we give a different approach. Denote $B=L^+\calG_\bba$ and $P=L^+\calG_\bbf$. Recall that $(LG/P)(k) = LG(k)/P(k)$, cf.\,e.g.\,\cite[Thm.\,1.4]{PR08}. The map $\iota\co BwB/B\to BwP/P$ is clearly surjective on $k$-points. To see that it is injective on $k$-points one uses the decomposition $P/B=\coprod_{w'\in W_\bbf}Bw'B/B$ (cf.~\cite[Rem 2.9]{Ri13}) together with the fact $l(w\cdot w')=l(w)+l(w')$ for all $w'\in W_\bbf$ which holds by the right $\bbf$-minimality. Hence, the map is bijective on $k$-points, and, using Zariski's Main Theorem, one can show it is enough to prove that $\iota$ is smooth. Let $w_0\in W_\bbf$ be the longest element. Consider the commutative diagram of $k$-schemes
\[
\begin{tikzpicture}[baseline=(current  bounding  box.center)]
\matrix(a)[matrix of math nodes, 
row sep=1.5em, column sep=2em, 
text height=1.5ex, text depth=0.45ex] 
{BwB\times^BBw_0B/B & BwP/B  \\ 
BwB/B& BwP/P. \\}; 
\path[->](a-1-1) edge node[above] {$\pi$}  (a-1-2);
\path[->](a-1-1) edge node[left] {$\on{pr}_1$}  (a-2-1);
\path[->](a-2-1) edge node[above] {$\iota$}  (a-2-2);
\path[->](a-1-2) edge node[right] {$\on{pr}$}  (a-2-2);
\end{tikzpicture}
\]
The inclusion $Bw_0B/B\subset P/B$ is an open immersion. As $w$ is assumed to be right $\bbf$-minimal, we have $l(w\cdot w_0)=l(w)+l(w_0)$, and standard properties of partial Demazure resolutions show that $\pi$ is an open immersion. We claim that both projections $\on{pr}_1$ and $\on{pr}$ are smooth and surjective; in that case  \cite[02K5]{StaPro} shows that the morphism $BwB/B\to BwP/P$ is smooth as well.  The assertion for both $\on{pr}$ and $\on{pr}_1$ results from the following general lemma applied with $\calP=\text{``smooth''}$ (using \'etale descent \cite[02VL]{StaPro}). We note that working with \'etale sheaves is justified by \cite[Thm.~1.4]{PR08} (cf.~also \cite[Prop.~A.4.9]{RS19} and \cite[pf.\,of Lem.\,4.9]{HRb}).
\end{proof}

\begin{lem} \label{pr_1_lem}
Let $S$ be a scheme. Let $X, Y$ be \'etale sheaves on ${\rm AffSch}/S$ \textup{(}affine schemes equipped with a map to $S$\textup{)}, and let $G$ be an \'etale sheaf of groups over $S$. Suppose $X$ \textup{(}resp.\,$Y$\textup{)} carries a right \textup{(}resp.\,left\textup{)} $G$-action such that $G$ acts freely on $X$. Let $\calP$ be a property of maps of algebraic spaces which is stable under base change, and which is \'etale local on the target. If $Y\to S$ is an algebraic space \textup{(}resp.~and has property $\calP$\textup{)}, then the canonical projection $\on{pr}_1\co X \times^G Y \to X/G$ is representable \textup{(}resp.~and has property $\calP$\textup{)}.
\end{lem}
\begin{proof}

We claim that there is a Cartesian diagram of \'etale sheaves
\[
\begin{tikzpicture}[baseline=(current  bounding  box.center)]
\matrix(a)[matrix of math nodes, 
row sep=1.5em, column sep=2em, 
text height=1.5ex, text depth=0.45ex] 
{X\times Y & X\times^GY  \\ 
X& X/G, \\}; 
\path[->](a-1-1) edge node[above] {}  (a-1-2);
\path[->](a-1-1) edge node[left] {}  (a-2-1);
\path[->](a-2-1) edge node[above] {}  (a-2-2);
\path[->](a-1-2) edge node[right] {$\on{pr}_1$}  (a-2-2);
\end{tikzpicture}
\]
where $X\to X/G$ and $X\times Y\to X$ are the canonical projections. The claim implies the representability of $\on{pr}_1$ by \cite[Cor.\,(1.6.3)]{LMB00} applied to the epimorphism $X\to X/G$ using that it is obvious for $X\times Y\to X$ (cf.~\cite[Rem.\,(1.5.1)]{LMB00}). Also this shows that if $Y\to S$ has property $\calP$, then $\on{pr}_1$ has property $\calP$. To prove the claim we note that the map
\[
X\times Y\longto X\times_{X/G,\on{pr}_1}(X\times^GY), \;\; (x,y)\mapsto (x,[x,y])
\]
is an isomorphism of \'etale sheaves. This is elementary to check using the freeness of the $G$-action on $X$.
\end{proof} 

\subsection{Proof of Proposition \ref{normal_prop} i)} \label{proof_i}

Consider a general Schubert variety $S_w(\bbf',\bbf)$ for some $\bbf',\bbf\subset \scrB(G,F)$, $w\in LG(k)$. By Corollary \ref{Reduction_Cor}, there exist a maximal $F$-split torus $S\subset G$, an alcove $\bbf\subset \bar{\bba}\subset \scrA(G,S,F)$, and an element $w_\aff\in W_\aff$ such that $S_w(\bbf',\bbf) \simeq S_{w_\aff}(\bba,\bbf)$ as $k$-schemes. Thus, we may and do assume that $\bbf'=\bba$ and that $w=w_\aff\in W_\aff$ is right $\bbf$-minimal by Lemma \ref{right_min_lem}. Further, if $\bbf=\bba$, then Proposition \ref{normal_prop} i) is proven in \cite[Prop.~9.7]{PR08}, and we explain how to extend the arguments to the general case. Let $\pi\co \Fl_{G,\bba}\to \Fl_{G,\bbf}$ be the canonical projection. Denote by $w_0\in w\cdot W_\bbf$ the maximal length representative inside the coset. Then $\pi^{-1}(S_{w}(\bba,\bbf))=S_{w_0}(\bba,\bba)$. Consider the normalization $\tilde{S}_w:=\tilde{S}_{w}(\bba,\bbf)\to S_{w}(\bba,\bbf)=:S_w$ (resp.~$\tilde{S}_{w_0}:=\tilde{S}_{w_0}(\bba,\bba)\to S_{w_{0}}(\bba,\bba)=:S_{w_0}$) which sits in a commutative diagram
\[
\begin{tikzpicture}[baseline=(current  bounding  box.center)]
\matrix(a)[matrix of math nodes, 
row sep=1.5em, column sep=2em, 
text height=1.5ex, text depth=0.45ex] 
{\tilde{S}_{w_0}& S_{w_0}  \\ 
\tilde{S}_w& S_w. \\}; 
\path[->](a-1-1) edge node[above] {}  (a-1-2);
\path[->](a-1-1) edge node[left] {$p$}  (a-2-1);
\path[->](a-2-1) edge node[below] {}  (a-2-2);
\path[->](a-1-2) edge node[right] {}  (a-2-2);
\end{tikzpicture}
\]
The right vertical map is an \'etale locally trivial fibration with typical fiber the homogenous space $Y:=L^+\calG_{\bbf}/ L^+\calG_{\bba}$ (cf.\,\cite[Lem.\,4.9]{HRb}), and thus the diagram is Cartesian because the property of being normal is local in the smooth topology \cite[0347]{StaPro}. By \cite[Prop.~9.7 a)]{PR08}, the top horizontal map is a finite birational universal homeomorphism, and so is the bottom horizontal map because all properties are fpqc local on the target \cite[02LA, 02L4, 0CEX]{StaPro}. Further, $\tilde{S}_{w_0}$ is Cohen-Macaulay by \cite[Prop.~9.7 d)]{PR08}, and hence so is $\tilde{S}_{w}$ by smoothness of $p$. For the property ``Frobenius split if $\on{char}(k)>0$'' (resp.\,``has rational singularities''), we note first that the counit of the adjunction
\begin{equation}\label{Ad_Map}
\calO_{\tilde{S}_w}\,\overset{\simeq}{\longto}\, p_*(p^*\calO_{\tilde{S}_w})\,=\,p_*(\calO_{\tilde S_{w_0}})
\end{equation}
is an isomorphism. Indeed, by \'etale descent for coherent sheaves we may argue locally in the \'etale topology. Using flat base change \cite[02KH]{StaPro}, it remains to prove \eqref{Ad_Map} for maps of type $\on{pr}\co Y\times X\to X$ for some $k$-variety $X$. As the push forward of $\calO_{Y}$ along the structure map $Y\to \Spec(k)$ is $\calO_{k}$ (cf.~ \cite[0AY8]{StaPro}), the assertion is immediate by flat base change.  Further, it is clear that $R^qp_*(\calO_{Y}) = 0$ when $q >0$, for example by a very special case of the Borel-Bott-Weil theorem ($Y$ is isomorphic to a classical flag variety over $k$, cf.\,\cite[\S 4.2.2]{HRa}). If $k$ is a field of characteristic $p>0$, then $\tilde S_{w_0}$ is Frobenius split by \cite[Prop.~9.7 c)]{PR08}, and the push forward of a splitting defines a splitting of $\tilde S_w$ by using \eqref{Ad_Map}. 

It remains to construct a rational resolution onto $\tilde S_w$. Recall that by \cite[Prop.~9.7 b)]{PR08} the natural closed immersion $S_w(\bba,\bba)\subset S_{w_0}(\bba,\bba)=S_{w_0}$ lifts to a closed immersion $\tilde S_w^\bba:=\tilde S_w(\bba,\bba)\subset \tilde S_{w_0}$. Let $p_w:= p|_{\tilde S_w^\bba}\co \tilde S_w^\bba\to \tilde S_w(\bba,\bbf)=\tilde S_w$ be the restriction. Fix a reduced decomposition $w=s_1\cdot\ldots\cdot s_n$ into simple reflections. The Demazure resolution $D(w)\to S_w^\bba$ factors through the normalization, and we claim that the composition 
\[
f\co D(w) \overset{\pi_w}{\longto} \tilde S_w^\bba \overset{p_w}{\longto} \tilde S_w
\]
is a rational resolution, i.e., $Rf_*\calO_{D(w)}$ is quasi-isomorphic to $\calO_{\tilde S_w}$ (in which case we call $f$ (cohomologically) trivial), and $R^q\om_{D(w)}=0$ for all $q >0$, where $\om_{D(w)}=\Om_{D(w)/k}^{n}$, $n=l(w)$. It is enough to show that $f$ is trivial. Indeed, if ${\rm char}(k) = 0$, $R^q\om_{D(w)}=0$ for $q>0$ would follow from the Grauert-Riemenschneider vanishing theorem; if ${\rm char}(k) = p>0$, as $D(w)$ is Frobenius split (cf.\,\cite[Prop 3.20]{Go01}) the vanishing would follow from the Grauert-Riemenschneider vanishing for Frobenius split varieties, cf.~\cite[Thm.~1.2]{MvdK92}. Further note that both morphisms $f=p_w\circ \pi_w$ are surjective and birational; for $p_w$ birational, use Lemma \ref{right_min_lem} and recall that $w$ is chosen to be right $\bbf$-minimal. Therefore $f$ is a resolution of singularities. Since $\tilde S_w$ is normal and integral, the Stein factorization of $f$ yields $f_* \calO_{D(w)} = \calO_{\tilde S_w}$. It remains to prove that $R^qf_* (\calO_{D(w)}) = 0$ for $q>0$.

Extend the reduced decomposition of $w$ to a reduced decomposition of $w_0$, and consider the following diagram
$$
\xymatrix{
D(w_0) \ar@/^/[drr]^{{\rm pr}_w} \ar@/_/[ddr]_{\pi_{w_0}} \ar@{-->}[dr]^{h} & & \\
& D^{\Box}(w) \ar[r]^q \ar[d]^g & D(w) \ar[d]^{f} \\
& \tilde S_{w_0} \ar[r]^{p} & \tilde S_w.
}
$$
Here the square is Cartesian, ${\rm pr}_w$ is the natural projection (onto the first $l(w)$ factors), and the dotted arrow $h$ exists because $f \circ {\rm pr}_w =  p \circ \pi_{w_0}$.  Since $\pi_{w_0}$ and $g$ are both birational, so is $h$. We claim that $g$ is trivial. By the Leray spectral sequence it suffices to prove that $\pi_{w_0}$ and $h$ are trivial. The triviality of $\pi_{w_0}$ is proven in \cite[Prop.\,9.7(d)]{PR08}. For $h$, note that $D^{\Box}_w=D_w\tilde{\times} Y$ is the twisted product, and likewise $D({w_0})=D(w)\tilde{\times}D(v)$ for the decomposition $w_0=w\cdot v$ with $v\in W_\bbf$. Under these identifications the map $h$ decomposes as $h=\id\tilde{\times} h_0$ where $\id\co D(w)\to D(w)$ is the identity, and $h_0\co D(v)\to Y$ is the Demazure resolution. Locally in the smooth topology on $D(w)^{\Box}$, the map $h$ is isomorphic to the direct product ${\rm id} \times h_0$. Using the vanishing of $R^q h_{0,*}(\calO_{D(v)})$ for $q>0$ ({\it loc.\,cit.} applied to $h_0$) and flat base change, we get the vanishing of $R^qh_*(\calO_{D(w_0)})$ for $q >0$. Also, $h_*(\calO_{D(w_0)}) = \calO_{D^{\Box}(w)}$ by the Stein factorization of $h$, as $D^{\Box}(w)$ is smooth and integral and $h$ is birational. This shows that $h$, hence $g$, is trivial. Now the required vanishing of $R^qf_*(\calO_{D(w)})$ for $q > 0$ follows from flat base change applied to the Cartesian square. This finishes the proof.

\subsection{Central extensions}\label{central_extension}
Let $\phi\co G'\to G$ be a map of (connected) reductive $F$-groups which induces an isomorphism on adjoint groups $G'_\ad\simeq G_\ad$ (or equivalently on simply connected groups $G'_\scon\simeq G_\scon$). Then $S':=\phi^{-1}(S)^o\subset \phi^{-1}(T)^o=:T'$ is a maximal $F$-split torus contained in a maximal torus. This induces a map on apartments $\scrA(G',S',F)\to \scrA(G,S,F)$ under which the facets correspond bijectively to each other. We denote the image of $\bbf$ by the same letter. The map $G'\to G$ extends to a map on parahoric group schemes $\calG':=\calG'_\bbf\to \calG_{\bbf}=:\calG$, and hence to a map on twisted partial affine flag varieties $\Fl_{G',\bbf}\to \Fl_{G,\bbf}$. We are interested in comparing their Schubert varieties.  

There is a natural map on Iwahori-Weyl groups 
\[
\tilde{W}'=W(G',S',F)\longto W(G,S,F)=\tilde{W},
\]
which is compatible with the action on the apartments $\scrA(G',S',F)\to \scrA(G,S,F)$. For $w'\in \tilde{W}'$ denote by $w\in \tilde{W}$ its image. As the map $\Fl_{G',\bbf}\to \Fl_{G,\bbf}$ is equivariant compatibly with the map $L^+\calG'\to L^+\calG$, we get a map of projective $k$-varieties
\begin{equation}\label{Schubertad}
S_{w'}=S_{w'}(\bba,\bbf)\longto S_{w}(\bba,\bbf)=S_{w}.
\end{equation}

\begin{prop}\label{Schubertadprop}
For each $w'\in \tilde{W}'$, the map \eqref{Schubertad} is a finite birational universal homeomorphism, and induces an isomorphism on the normalizations.
\end{prop}

We need some preparation. The Iwahori-Weyl groups are equipped with a Bruhat order $\leq$ and a length function $l$ according to the choice of $\bba$. 

\begin{lem} \label{IWad}
The map $\tilde{W}'\to \tilde{W}$ induces an isomorphism of affine Weyl groups compatible with the simple reflections, and thus compatible with $\leq$ and $l$.
\end{lem}
\begin{proof} 
Let $\tilde{W}_{\rm sc}$ be the Iwahori-Weyl group associated to the simply connected cover $\psi: G_{\rm sc} \to G_{\rm der}$ and the torus $S_{\rm sc} := \psi^{-1}(S \cap G_{\rm der})^o$. By \cite[5.2.10]{BT84}, $\tilde{W}_{\rm sc}$ can be identified with the affine Weyl group for $(G, S, \bba)$ as well as $(G', S', \bba)$, cf.\,also \cite[Prop.\,13]{HR08}.
\end{proof}

We denote the affine Weyl $W_\aff$ of $\tilde{W}$, resp.~$\tilde{W}'$ by the same symbol.

\begin{cor} \label{Setad}
For each $w'\in \tilde{W}'$, the map $\tilde{W}'\to \tilde{W}$ induces a bijection
\[
\{v'\in \tilde{W}'\;|\; v'\leq w'\}\overset{\simeq}{\longto} \{v\in \tilde{W}\;|\; v\leq w\}
\]
under which the right $\bbf$-minimal elements correspond to each other.
\end{cor}
\begin{proof} Write $w'=\tau'\cdot w'_1$ according to $\tilde{W}'=\Om'\ltimes W_\aff$. After left translation by $(\dot{\tau}')^{-1}$ for any representative $\dot{\tau}'\in \on{Norm}_{G'}(S')(F)$ of $\tau'$, we may assume $\tau'=1$. Lemma \ref{IWad} implies the corollary. 
\end{proof}

\begin{lem} \label{openSchubertad}
For each $w'\in \tilde{W}'$, the map 
\[
L^+\calG'_{\bba}\cdot{w}'\cdot L^+\calG'/L^+\calG'\,\longto\, L^+\calG_{\bba}\cdot {w}\cdot L^+\calG/L^+\calG
\] 
is an isomorphism.
\end{lem}
\begin{proof} 
By the proof of Corollary \ref{Setad}, we may assume that $w'\in W_\aff$, and that $w'$ is right $\bbf$-minimal. By Lemma \ref{right_min_lem}, we may further assume that $\bbf=\bba$ is the alcove. Let $w'=s'_1\cdot\ldots\cdot s'_n$ be a reduced decomposition into simple reflections. By Lemma \ref{IWad}, the decomposition $w=s_{1}\cdot\ldots\cdot s_{n}$ is a reduced decomposition of $w$. Let $\pi_{w'}\co D(w')\to S_{w'}$ (resp.~$\pi_{w}\co D(w)\to S_{w}$) be the Demazure resolution associated with the reduced decomposition, cf. \cite[Prop 8.8]{PR08}. There is a commutative diagram of $k$-schemes
\[
\begin{tikzpicture}[baseline=(current  bounding  box.center)]
\matrix(a)[matrix of math nodes, 
row sep=1.5em, column sep=2em, 
text height=1.5ex, text depth=0.45ex] 
{D(w')&D(w)  \\ 
S_{w'}& S_{w}, \\}; 
\path[->](a-1-1) edge node[above] {}  (a-1-2);
\path[->](a-1-1) edge node[left] {$\pi_{w'}$}  (a-2-1);
\path[->](a-2-1) edge node[below] {}  (a-2-2);
\path[->](a-1-2) edge node[right] {$\pi_{w}$}  (a-2-2);
\end{tikzpicture}
\]
where the vertical maps are birational and isomorphisms onto the open cells. Hence, it is enough to show that $D(w')\to D(w)$ is an isomorphism. By induction on $l(w')=n$, we reduce to the case $w'=s'$ (and hence $w=s$) is a simple reflection. In this case, $\pi_{w'}$ and $\pi_{w}$ are isomorphisms, and we have to show that $\bbP^1_k\simeq S_{w'}\to S_{w}\simeq \bbP^1_k$ is an isomorphism. The crucial observation is now that the map $\calG'\to \calG$ on Bruhat-Tits groups schemes is the identity on the $\calO_F$-extension of the root subgroups (cf.\,\cite[4.6.3, 4.6.7]{BT84}). Hence, the map $S_{w'}\to S_{w}$ restricted to the open cells $\bbA^1_k\subset \bbP^1_k$ is the identity. The lemma follows. 
\end{proof}

\begin{proof}[Proof of Proposition \ref{Schubertadprop}]
The $ L^+\calG'_{\bba}$-orbits (resp. $ L^+\calG_{\bba}$-orbits) in $S_{w'}$ (resp. $S_{w}$) correspond under the map $S_{w'}\to S_{w}$ bijectively to each other, cf.~Corollary \ref{Setad}. Hence, Lemma \ref{openSchubertad} implies that the map $S_{w'}\to S_{w}$ is birational and bijective on $k$-points. As being quasi-finite and proper implies finite, the map in question must be finite. To see that the map is a universal homeomorphism consider the commutative diagram of $k$-schemes 
\[
\begin{tikzpicture}[baseline=(current  bounding  box.center)]
\matrix(a)[matrix of math nodes, 
row sep=1.5em, column sep=2em, 
text height=1.5ex, text depth=0.45ex] 
{\tilde{S}_{w'}&\tilde{S}_{w}  \\ 
S_{w'}& S_{w}, \\}; 
\path[->](a-1-1) edge node[above] {}  (a-1-2);
\path[->](a-1-1) edge node[left] {}  (a-2-1);
\path[->](a-2-1) edge node[below] {}  (a-2-2);
\path[->](a-1-2) edge node[right] {}  (a-2-2);
\end{tikzpicture}
\]
where the vertical maps are the normalization morphisms. By Proposition \ref{normal_prop} i), the vertical maps are finite birational universal homeomorphisms. In particular, the map $\tilde{S}_{w'}\to \tilde{S}_{w}$ is a birational bijective proper, hence finite, morphism of normal varieties, and therefore it is an isomorphism. 
This shows that the map $S_{w'}\to S_{w}$ is a universal homeomorphism, and the proposition follows. 
\end{proof}

\subsection{Simple reduction}\label{Absolutely_Simple_Red} There is a finite index set $J$, and an isomorphism of $F$-groups
\[
G_\scon\,=\, \prod_{J\in J}\Res_{F_j/F}(G_j),
\]
where each $F_j/F$ is a finite separable extension, and each $G_j$ is an absolutely almost simple, simply connected $F_j$-group. Under the identification of buildings $\scrB(G_\scon,F)=\prod_{j\in J}\scrB(G_j,F_j)$ the facet $\bbf$ corresponds to facets $\bbf_j\subset \scrB(G_j,F_j)$ for each $j\in J$.

\begin{lem} \label{simply_connected_variety}
There is an isomorphism of $k$-ind-schemes
\[
\Fl_{G_\scon,\bbf}\,\simeq\, \prod_{j\in J}\Fl_{G_j,\bbf_j}
\]
under which the Schubert varieties correspond isomorphically to each other. 
\end{lem}
\begin{proof}
It is enough to treat the following two cases separatedly.\smallskip\\
{\em Products}. If $G=G_1\times G_2$ is a direct product of two $F$-groups, then we have a direct product of affine Weyl groups $W_\aff=W_{\aff,1}\times W_{\aff,2}$. Now, for each $w=(w_1,w_2)\in W_\aff$, there is an equality on Schubert varieties $S_w=S_{w_1}\times S_{w_2}$ which is easy to prove. In particular, if both $S_{w_1}$ and $S_{w_2}$ are normal, then $S_w$ is normal by \cite[06DG]{StaPro}.\smallskip\\
{\em Restriction of scalars}. \label{Restriction_subsec} Let $G=\Res_{F'/F}(G')$ where $F'/F$ is a finite separable extension, and $G'$ is an $F'$-group. By \cite[Prop 4.7]{HRb}, we have $\calG_\bbf=\Res_{\calO_{F'}/\calO_F}(\calG'_\bbf)$ where we use the identification $\scrB(G,F)=\scrB(G',F')$. Now choose\footnote{The result is independent of the choice of the uniformizers $u$, resp.~$t$ because loop groups can be defined without a reference to them, cf.~\cite[\S(1.3) footnote 2]{BL94} (or \cite[\S 2]{Ri13}).} a uniformizer $u\in \calO_{F'}$. Since $k$ is algebraically closed, we have $\calO_{F'}=k\pot{u}$ resp.~$F'=k\rpot{u}$. For any $k$-algebra $R$, we have $R\pot{t}\otimes_{\calO_{F}}\calO_{F'}=R\pot{u}$ resp.~$R\rpot{t}\otimes_{{F}} {F'}=R\rpot{u}$. This gives an equality on loop groups $L^+\calG_\bbf=L^+\calG'_\bbf$ resp.~$LG=LG'$. Hence, there is an equality on twisted affine flag varieties $\Fl_{G,\bbf}=\Fl_{G',\bbf}$ under which the Schubert varieties correspond to each other. 
\end{proof}

\subsection{Proof of Proposition \ref{normal_prop} ii)} \label{proof_ii}
Consider a general Schubert variety $S_w(\bbf',\bbf)$ for some $\bbf',\bbf\subset \scrB(G,F)$, $w\in LG(k)$. By Corollary \ref{Reduction_Cor}, there exist a maximal $F$-split torus $S\subset G$, an alcove $\bbf\subset \bar{\bba}\subset \scrA(G,S,F)$, and an element $w_\aff\in W_\aff$ such that $S_w(\bbf',\bbf) \simeq S_{w_\aff}(\bba,\bbf)$ as $k$-schemes. Proposition \ref{Schubertadprop} applied to $G_\scon\to G$ shows that the normalization $\tilde S_{w_\aff}(\bba,\bbf)$ is isomorphic to the normalization of a Schubert variety inside 
\[
\Fl_{G_\scon,\bbf}\,\simeq\, \prod_{j\in J}\Fl_{G_j,\bbf_j},
\]
cf.~Lemma \ref{simply_connected_variety}. 
Thus, $\tilde S_{w_\aff}(\bba,\bbf)\simeq \prod_{j\in J} \tilde S_{w_j}(\bba_j,\bbf_j)$ for some $w_j\in LG_j(k)$. Now assume that $S_w\simeq S_{w_\aff}(\bba,\bbf)=:S_{w_\aff}$ is normal. It remains to prove that each variety $S_{w_j}:=S_{w_j}(\bba_j,\bbf_j)$ is normal as well. First note that the canonical map $\prod_{j\in J}S_{w_j}\to S_{w_\aff}$ must be an isomorphism by Proposition \ref{Schubertadprop}, so that the product of the varieties $S_{w_j}$ is normal. Fix some $j_0\in J$, and consider
\[
\tilde{S}_{w_{j_0}}\times \prod_{j\not = j_0}S_{w_j}\to S_{w_{j_0}}\times\prod_{j\not = j_0}S_{w_j},
\]
where $\tilde{S}_{w_{j_0}}\to S_{w_{j_0}}$ denotes the normalization. By Proposition \ref{normal_prop} i), this map is finite and birational. As the target is normal, it must be an isomorphism, so that  $\tilde{S}_{w_{j_0}}\to S_{w_{j_0}}$ must be an isomorphism (because being an isomorphism is fpqc local on the target). This proves Proposition \ref{normal_prop} ii). \qed

\section{Reducedness of special fibers of local models}

\subsection{Weil restricted local models in mixed characteristic}

We now switch back to the notation of \S\ref{main_result_sec}. We first treat the case where $F/\bbQ_p$ is a mixed characteristic local field. Recall that in this case we are assuming $G=\Res_{K/F}(G_1)$ where $K/F$ is a finite extension with residue field $k$ of characteristic $p>0$, and $G_1$ is a tamely ramified connected reductive $K$-group. To simplify the discussion and notation we first assume that $K/F$ is totally ramified. 
The extension to the more general case is easy and is explained in Remark \ref{extension_general_rmk} below.
We fix a uniformizer $\varpi\in K$, and let $Q\in \calO_F[u]$ be its minimal polynomial (an Eisenstein polynomial). The reader who is only concerned with Pappus-Zhu local models may take $K=F$ and $G_1 = G$ throughout this discussion.

Under the identification $\scrB(G,F)=\scrB(G_1,K)$ (cf.~\cite[Prop.~4.6]{HRb}), the facet $\bbf$ corresponds to a facet denoted $\bbf_1$. We denote by $\calG_1=\calG_{\bbf_1}$ the parahoric $\calO_K$-group scheme of $G_1$ associated with $\bbf_1$. Then $\calG:=\Res_{\calO_K/\calO_F}(\calG_{1})$ is the parahoric $\calO_F$-group scheme of $G$ associated with $\bbf$, cf.~\cite[Cor.~4.8]{HRb}. We let $A_1\subset G_1$ be a maximal $K$-split torus whose apartment $\scrA(G_1,A_1,K)$ contains $\bbf_1$.

\subsubsection{Recollections}\label{recollections}
Following the method of \cite[\S3]{PZ13}, a connected reductive $\calO_F[u^\pm]$-group scheme $\underline{G}_1$ is constructed in \cite[\S3.1, Prop.~3.3]{Lev16} (cf.~also \cite[Prop.~4.10 i)]{HRb}) together with an isomorphism
\begin{equation}\label{spread}
\underline{G}_1\otimes_{\calO_F[u^\pm],u\mapsto \varpi}K\;\simeq\; G_1.
\end{equation}
We let $\underline{A}_1\subset \underline{G}_1$ be the split $\calO_F[u^\pm]$-torus extending $A_1\subset G_1$. By \cite[\S 4.1.3]{PZ13}, the isomorphism \eqref{spread} induces identifications of apartments
\begin{equation}\label{kappa_iso}
\scrA(G_1,A_1,K)\;\simeq\; \scrA(\underline{G}_{1,\kappa\rpot{u}},\underline{A}_{1,\kappa\rpot{u}}, \kappa\rpot{u}),
\end{equation}
for $\kappa=k,F$. We denote by $\bbf_{1,\kappa\rpot{u}}$ the facet corresponding to $\bbf_1$. By \cite[Thm.~4.1]{PZ13} (cf.~also \cite[Thm.~3.3.3]{Lev16}), there exists a unique $\calO_F[u]$-group scheme $\underline{\calG}_1$ with the following three properties: a) the restriction of $\underline{\calG}_1$ to $\calO_F[u^\pm]$ is $\underline{G}_1$; b) the base change $\underline{\calG}_1\otimes_{\calO_F[u],u\mapsto \varpi}\calO_K$ is $\calG_1$; c) the base change $\underline{\calG}_1\otimes_{\calO_F[u]}\kappa\rpot{u}$ is the parahoric $\kappa\pot{u}$-group scheme associated with the facet $\bbf_{1,\kappa\rpot{u}}$ under \eqref{kappa_iso} for $\kappa=k,F$. Note that the group scheme $\underline{\calG}_1$ is already uniquely determined by property a) and property c) with $\kappa=F$ by \cite[4.2.1]{PZ13}. 

Following \cite[Def 4.1.1]{Lev16} (see also \cite[\S4.4]{HRb} for how this fits into the general picture of Beilinson-Drinfeld Grassmannians), we define the symbol $\Gr_{\calG}$ to be the functor on the category of $\calO_F$-algebras $R$ given by the isomorphism classes of tuples $(\calF,\al)$ with
\begin{equation}\label{BD_Grass_dfn}
\begin{cases}
\text{$\calF$ a $\underline{\calG}_1$-torsor on $\Spec(R[u])$};\\
\text{$\al\co \calF|_{\Spec(R[u][\nicefrac{1}{Q}])}\simeq \calF^0|_{\Spec(R[u][\nicefrac{1}{Q}])}$ a trivialization},
\end{cases}
\end{equation}
where $\calF^0$ denotes the trivial torsor. If $Q=u-\varpi$, i.e., $K=F$, then $\Gr_{\calG}$ is the BD-Grassmannian defined in \cite[6.2.3; (6.11)]{PZ13}. Informally, we think about $\Gr_{\calG}$ as being the Beilinson-Drinfeld Grassmannian associated with the parahoric $\calO_F$-group scheme $\calG$.

By \cite[Thm.~4.2.11]{Lev16} (cf.\,also \cite[Thm.~4.16]{HRb}), the functor $\Gr_{\calG}$ is representable by an ind-projective ind-scheme over $\calO_F$. Its generic fiber is equivariantly (for the left action of the loop group) 
isomorphic to the usual affine Grassmannian $\Gr_{G}$ formed using the parameter $z:=u-\varpi\in K[u]$. Its special fiber is canonically isomorphic to the twisted affine flag variety $\Fl_{G^\flat,\bbf^\flat}$ where we denote
\begin{equation} \label{Gflat}
G^\flat:=\underline{G}_1\otimes_{\calO_F[u^\pm]} k\rpot{u}, \;\;\;\bbf^\flat:=\bbf_{1,k\rpot{u}}.
\end{equation}
Informally, we think about the $k\rpot{u}$-group $G^\flat$ as being a connected reductive group of the ``same type'' as the $K$-group $G_1$. By the discussion above, it is equipped with an identification of apartments $\scrA(G^\flat,A^\flat,k\rpot{u})=\scrA(G_1,A_1,K)$ where $A^\flat:=\underline{A}_1\otimes k\rpot{u}$.

Recall we fixed a conjugacy class $\{\mu\}$ of geometric cocharacters in $G$ with reflex field $E/F$. This defines a closed subscheme $\Gr_{G}^{\leq \{\mu\}}\subset \Gr_{G}\otimes_F E$ inside the affine Grassmannian which is a (geometrically irreducible) projective $E$-variety. Following \cite[Def 7.1]{PZ13} and \cite[Def 4.2.1]{Lev16} (cf.\,also  \cite[Prop 4.2.4]{Lev16}), the {\em local model} $M_{\{\mu\}}=M(\uG_1,\calG_\bbf,\{\mu\},\varpi)$ is the scheme theoretic closure of the locally closed subscheme 
\[
\Gr_{G}^{\leq\{\mu\}}\,\hookto\, \Gr_{G}\otimes_F E\,\hookto\, \Gr_{\calG}\otimes_{\calO_F}\calO_E.
\]
By definition, the local model is a reduced flat projective $\calO_E$-scheme, and equipped with an embedding of its special fiber
\[
\overline{M}_{\{\mu\}}:=M_{\{\mu\}}\otimes_{\calO_E}k_E\;\hookto\; \Fl_{G^\flat,\bbf^\flat}\otimes_kk_E.
\]
We now define the admissible locus $\calA(G,\{\mu\})\subset \Fl_{G^\flat,\bbf^\flat}\otimes_k\bar{k}$. Recall from \cite[\S5.4]{HRb} that there are identifications of Iwahori-Weyl groups
\[
W=W(G, A, \breve{F})=W(G_1,A_1,\breve{K})=W(G^\flat,A^\flat,\bar{k}\rpot{u}).
\]
Then the admissible locus $\calA(G,\{\mu\})$ is defined as the union of the $(\bbf^\flat,\bbf^\flat)$-Schubert varieties $S_w\subset \Fl_{G^\flat,\bbf^\flat}\otimes_k\bar{k}$ where $w$ runs through the elements of the admissible set $\Adm_{\{\mu\}}^\bbf\subset W_\bbf\backslash W/W_\bbf$. Under the assumption that all Schubert varieties inside $\calA(G,\{\mu\})$ are normal, we show in the next subsection that $\overline{M}_{\{\mu\}}\otimes \bar{k}=\calA(G,\{\mu\})$.

\subsubsection{Proof of Theorem \ref{special_fiber_thm} i\textup{)} in mixed characteristic}\label{special_fibers} The proof uses the Coherence Conjecture \cite{Zhu14}, and then follows easily from the method in \cite[\S 9.2.2]{PZ13} using \cite[Thm.~4.3.2]{Lev16} and Proposition \ref{normal_prop}.

Note that the inclusion $\calA(G,\{\mu\})\subseteq \overline{M}_{\{\mu\}}\otimes_{k_E}\bar{k}$ is proven as part of \cite[Thm.~5.14]{HRb}. Let $\calV_1=\Lie \,\underline{\calG}_1$ denote the Lie algebra, which is a free $\calO_F[u]$-module of rank $\dim_K(G_1)$. The adjoint representation $\underline{\calG}_1\to \GL(\calV_1)$ induces by functoriality a morphism of ind-projective $\calO_F$-ind-schemes
\[
\ad\co \Gr_{\calG}\;\to\; \Gr_{\GL(\calV_1)},
\]
where the target is defined as in \eqref{BD_Grass_dfn} using the $\calO_F[u]$-group scheme $\GL(\calV_1)$. 

Let $\calL_{\det}$ be the determinant line bundle on the target, and denote by $\calL:=\ad^*(\calL_{\det})$ its pullback.
Let $\calL_{\sF}$ (resp.~$\calL_{\bar{k}}$) denote the restriction of $\calL$ to the geometric generic (resp.~geometric special) fiber $\Gr_{\calG,\sF}=\Gr_{G,\sF}$ (resp.~$\Gr_{\calG,\bar{k}}=\Fl_{G^\flat,\bbf^\flat}\otimes \bar{k}$). 

\begin{lem} \label{line_bundle}
The pullback of the line bundle $\calL$ to $M_{\{\mu\}}$ is relatively ample over $\calO_E$, and for all $n\geq 1$ one has
\begin{equation}\label{Coherence}
\dim_{\bar F} \Ga\big(M_{\{\mu\},\sF},\calL_\sF^{\otimes n}\big)\;=\; \dim_{\bar k}\Ga\big(\calA(G,\{\mu\})_{\bar{k}},\calL_{\bar{k}}^{\otimes n}\big).
\end{equation}
\end{lem}
\begin{proof} This lemma is a direct consequence of the Coherence Conjecture \cite{Zhu14} invoking Proposition \ref{normal_prop} ii) and the assumption on the normality of Schubert varieties, see also \cite[\S9]{PZ13} (resp.~\cite[Thm.~4.3.2]{Lev16}) for similar arguments. We recall the argument for convenience.

First note that since $M_{\{\mu\}}\to \Spec(\calO_E)$ is proper, the line bundle $\calL$ is relatively ample on $M_{\{\mu\}}$ if and only if its fibers $\calL_{F}$, $\calL_k$ are ample [$\on{EGA IV}_3$, Cor.~9.6.4] if and only if its geometric fibers $\calL_{\bar{F}}$, $\calL_{\bar{k}}$ are ample \cite[01VR (4)]{StaPro}. Hence, the statement of the lemma only depends on the geometric fibers. 

By Proposition \ref{normal_prop} ii), using the normality assumption on $M_{G,\{\mu\},\sF}$ (resp.~the maximal $(\bbf^\flat,\bbf^\flat)$-Schubert varieties inside $\calA(G,\{\mu\})$), there is an isomorphism
\begin{equation}\label{schemes_fibers_2}
M_{G,\{\mu\},\sF}\simeq M_{G,\{\mu\},\sF}^o \;\;\;\;\text{(resp.}\;\;  \calA(G,\{\mu\})\simeq \calA(G,\{\mu\})^o\;\textup{)}
\end{equation}
where, in a change of notation, the superscript $o$ denotes as in \cite[\S 2.2]{Zhu14} (see also \cite[(9.18) ff.]{PZ13}) the translated to the neutral component (Iwahori-) Schubert variety inside the affine Grassmannian (resp.~twisted affine flag variety) for the simply connected group. Note that $M_{G,\{\mu\},\sF}^o$ (resp.~$\calA(G,\{\mu\})^o$) only depends on the adjoint group $G_\ad$ and the image of $\{\mu\}$ under $G\to G_\ad$: the simply connected groups of $G$ and $G_\ad$ are the same, and the translated to the affine Weyl group admissible set depends only on the image of $\{\mu\}$ under $G\to G_\ad$, cf.~Corollary \ref{Setad}. This procedure is also compatible with the formation of $\calL$ because the morphism $\on{ad}$ factors through the adjoint group. We now study the decomposition of $M_{G,\{\mu\},\sF}^o$ (resp.~$\calA(G,\{\mu\})^o$) in terms of the simple factors of $G_\ad$.  

Let $G_\ad=\prod_{j\in J}\Res_{F_j/F}(G_{ j, \ad})$ where $J$ is a finite index set, $F_j/F$ are finite field extensions containing $K$ and $G_{j,\ad}$ are absolutely simple, tamely ramified $F_j$-groups. The geometric generic fiber becomes
\begin{equation}\label{decom_schubert}
M_{G,\{\mu\},\sF}^o\simeq \prod_{j\in J} \prod_{i=1}^{n_j} \Gr_{G_{j,\ad},\bar F}^{\leq \{\mu_i^{(j)}\}, o}\;\subset\; \Gr_{G_\scon,\bar F}
\end{equation}
where $n_j:=[F_j:F]$ and $(\{\mu^{(j)}_i\})_{i=1,\ldots,n_j}$ is the $j$-factor of $\{\mu\}$ under $G_\sF\to G_{\ad,\sF}$. Note that \eqref{decom_schubert} comes from the analogous decomposition for $\Gr_{G_\scon,\bar F}$. For $j\in J$, consider the maximal unramified subextension $F_j/F_j^{\on{ur}}/F$, and write $n_j=m_j\cdot l_j$ with $m_j:=[F_j:F_j^{\on{ur}}]$ and $l_j:=[F_j^{\on{ur}}:F]$. Also for each $j\in J$ we reorder 
\[
(\{\mu^{(j)}_i\})_{i=1,\ldots,n_j}=\big((\{\mu_{i,k}^{(j)}\})_{k=1,\ldots,m_j},\ldots,(\{\mu_{i,k}^{(j)}\})_{k=1,\ldots,m_j}\big)_{i=1,\ldots, l_j}
\]
according to $\Res_{F_j/F}(G_{j,\ad})=\Res_{F_j/F_j^{\on{nr}}}(\Res_{F_j^{\on{nr}}/F_j}(G_{j,\ad}))$. As $F_j/F_j^{\on{ur}}$ is totally ramified, we obtain in the geometric special fiber\footnote{Note that it is clear how to add conjugacy classes of $1$-parameter subgroups, by choosing dominant representatives (for any notion of dominant) and taking the conjugacy class of the sum. This is independent of all choices.}
\begin{equation}\label{decom_adm_locus}
\calA(G,\{\mu\})^o\simeq \prod_{j\in J}\prod_{i=1}^{l_j}\calA(G_{j,\ad},\{\mu_{i,1}^{(j)}\}+\ldots \{\mu_{i,m_j}^{(j)}\})^o.
\end{equation}
Again \eqref{decom_adm_locus} comes from the analogous decomposition for $\Fl_{G^\flat_{\scon},\bbf^\flat_{\scon}}$ where $\bbf^\flat_\scon\subset \scrB(G_\scon^\flat,k\rpot{u})$ denotes the facet corresponding to $\bbf^\flat$. Here one has to remember $G=\Res_{K/F}(G_1)$ and that the composite field $K\cdot F_j^{\on{ur}}$ is a subfield of $F_j/F_j^{\on{ur}}$.

 Also $\calL_{\bar{F}}$ (resp.~$\calL_{\bar{k}}$) decomposes according to \eqref{decom_schubert} (resp.~\eqref{decom_adm_locus}), so that its ampleness on $M_{G, \{\mu\},\sF}^o$ (resp.~on $\calA(G,\{\mu\})$) now follows from the explicit formula given in \cite[Lem.~4.2]{Zhu} (see also \cite[Prop.~4.3.6]{Lev16}). Note that we are using here that a line bundle is ample if and only if its restriction to the reduced locus is ample, cf.\,\cite[Prop 4.5.13]{EGA2}. For each $j\in J$, the remaining claim \eqref{Coherence} now reads
\begin{equation}\label{last_product}
\prod_{i=1}^{n_j} \dim_{\bar F} \Ga\big(\Gr_{G_{j,\ad},\bar F}^{\leq \{\mu_i^{(j)}\}, o}, \calL^{\otimes n}_{\bar F}\big)\;=\; \prod_{i=1}^{l_j}\dim_{\bar k} \Ga\big(\calA(G_{j,\ad},\{\mu_{i,1}^{(j)}\}+\ldots \{\mu_{i,m_j}^{(j)}\})^o, \calL^{\otimes n}_{\bar k}\big).
\end{equation}
By \cite[Prop.~4.3.8]{Lev16} (and the references cited there), we have the product formula
\[
\prod_{k=1}^{m_j} \dim_{\bar F} \Ga\big(\Gr_{G_{j,\ad},\bar F}^{\leq \{\mu_{i,k}^{(j)}\}, o}, \calL^{\otimes n}_{\bar F}\big)\;=\; \dim_{\bar F} \Ga\big(\Gr_{G_{j,\ad},\bar F}^{\leq \{\mu_{i,1}^{(j)}\}+\ldots+\{\mu_{i,m_j}^{(j)}\}, o}, \calL^{\otimes n}_{\bar F}\big).
\]
Thus \eqref{last_product} follows from the equality
\begin{equation}\label{final_step}
\dim_{\bar k} \Ga\big(\calA(G_{j,\ad},\{\mu_{i,1}^{(j)}\}+\ldots \{\mu_{i,m_j}^{(j)}\})^o, \calL^{\otimes n}_{\bar k}\big) =   \dim_{\bar F} \Ga\big(\Gr_{G_{j,\ad},\bar F}^{\leq \{\mu_{i,1}^{(j)}\}+\ldots+\{\mu_{i,m_j}^{(j)}\}, o}, \calL^{\otimes n}_{\bar F}\big)
\end{equation}
which follows from the main theorem of \cite{Zhu14}.
\end{proof}

Lemma \ref{line_bundle} is enough to conclude the proof of Theorem \ref{special_fiber_thm} in this case, as follows. As in \cite[\S9]{PZ13}, by the local constancy of the Euler characteristic \cite[Thm.~7.9.4]{EGA32} and Serre's cohomology vanishing theorem \cite[Thm.~2.2.1]{EGA31},
we have for $n>\!\!>0$,
\[
\dim_{\bar{F}} \Ga(M_{\{\mu\},\sF},\calL_\sF^{\otimes n})=\dim_{\bar k} \Ga(\overline{M}_{\{\mu\},\bar{k}},\calL_{\bar{k}}^{\otimes n})\;\geq\; \dim_{\bar k}\Ga(\calA(G,\{\mu\})_{\bar{k}},\calL_{\bar{k}}^{\otimes n}),
\]
and thus equality. As $\calL_{\bar{k}}$ is ample, this implies $\overline{M}_{\{\mu\},\bar{k}}=\calA(G,\{\mu\})$, and finishes the proof of Theorem \ref{special_fiber_thm} i) in this case.

\begin{rmk}\label{extension_general_rmk}
Let $K/F$ be any finite field extension, not necessarily totally ramified. 
We explain how the preceding discussion extends to this more general case.
Denote by $K_0/F$ the maximal unramified subextension of $K/F$ with residue field $k_0/k$. 
We now have $Q\in \calO_{K_0}[u]$ for the Eisenstein polynomial of $\varpi$. We have the parahoric $\calO_K$-group scheme $\calG_1 = \calG_{\bf f_1}$ as above, and we define now the parahoric $\calO_{K_0}$-group scheme $\calG := \Res_{\calO_{K}/\calO_{K_0}}(\calG_1)$ and its associated positive loop group $L^+_0 \calG$ over $\calO_{K_0}$.
As in \cite[$\S4.3$]{HRb}, the group scheme $\uG_1$ in \eqref{spread} (resp.~$\ucG_1$) is now defined over $\calO_{K_0}[u^\pm]$ (resp.~$\calO_{K_0}[u]$).
Hence, the Beilinson-Drinfeld Grassmannian in \eqref{BD_Grass_dfn} is defined over $\calO_{K_0}$ as well.
This ind-projective ind-scheme is denoted by $\Gr_{\calG,0}\to \Spec(\calO_{K_0})$.
As in \cite[\S4.1]{Lev16} we define the $\calO_F$-ind-scheme
\[
\Gr_{\calG}\defined \Res_{\calO_{K_0}/\calO_F}(\Gr_{\calG, 0}),
\]
which is ind-projective as well, and carries a natural left action of the positive loop group $L^+\calG := \Res_{\calO_{K_0}/\calO_F}(L^+_0 \calG)$.
Again its generic fiber is isomorphic to the usual affine Grassmannian $\Gr_G$, equivariantly for the action of $(L^+\calG)_\eta = L^+_zG$, where $z = u - \varpi$.
Its special fiber is isomorphic to $\Fl_{G^\flat,\bbf^\flat}$ where now
\[
G^\flat:=\Res_{k_0\rpot{u}/k\rpot{u}}\big(\underline{G}_1\otimes_{\calO_{K_0}[u^\pm]} k_0\rpot{u}\big), 
\]
and $\bbf^\flat\subset\scrB(G^\flat,k\rpot{u})$ corresponds to $\bbf_{1}$ under $\scrB(G^\flat,k\rpot{u})=\scrB(\uG_{1},k_0\rpot{u})=\scrB(G_1,K)$. 

For a geometric conjugacy class $\{\mu\}$ in $G$, the local model $M_{\{\mu\}}=M(\uG_1,\calG_\bbf,\{\mu\},\varpi)$ is the scheme theoretic closure of $\Gr_{G}^{\leq\{\mu\}}\subset \Gr_{\calG}\otimes_{\calO_F}\calO_E$. Let $E_0$ denote the compositum of $K_0$ with the reflex field $E/F$ of $\{\mu\}$. Then according to $G=\Res_{K_0/F}(\Res_{K/K_0}(G_1))$ the conjugacy class $\{\mu\}$ decomposes as a tuple of $\Res_{K/K_0}(G_1)$-conjugacy classes $\{\mu_j\}$, $1\leq j\leq [K_0:F]$ each having reflex field $E_0/K_0$.
Since $\calO_{K_0}/\calO_F$ is \'etale, we have as $\calO_{K_0}$-ind-schemes
\[
\Gr_\calG\otimes_{\calO_F}\calO_{K_0}\;\simeq\;\prod_{1\leq j\leq [K_0:F]}\Gr_{\calG,0},
\]
so that $M_{\{\mu\}}\otimes_{\calO_E}\calO_{E_0}\simeq \prod_j M_{\{\mu_j\}}$ where each $M_{\{\mu_j\}}$ is a local model for $\Res_{K/K_0}(G_1)$.
Further, on each $\Gr_{\calG,0}$ we have the line bundle $\calL_{0}$ constructed as in \eqref{special_fibers} which induces a line bundle $\calL$ on $\Gr_\calG$ by \'etale descent. 
Now the analogue of Lemma \ref{line_bundle} for the pair $(M_{\{\mu\}},\calL)$ is immediate from the product decomposition $M_{\{\mu\}}\otimes_{\calO_E}\calO_{E_0}\simeq \prod_j M_{\{\mu_j\}}$ and the validity of Lemma \ref{line_bundle} for each pair $(M_{\{\mu_j\}},\calL_{0})$.
\end{rmk}

\subsection{Equal characteristic local models}
We now treat the case where $F\simeq k\rpot{t}$ is of equal characteristic. In this situation, we assume that in the simply connected group $G_\scon\simeq \prod_{j\in J}\Res_{F_j/F}(G_j)$ each absolutely almost simple factor $G_j$ splits over a tamely ramified extension of $F_j$. We also fix a uniformizer $t\in \calO_F$ so that $\calO_F=k\pot{t}$ . Let $\calG=\calG_\bbf$ denote the parahoric $k\pot{t}$-group scheme.

\subsubsection{Recollections} Similarly to \eqref{BD_Grass_dfn} the Beilinson-Drinfeld affine Grassmannian $\Gr_\calG$ is the functor on the category of $k\pot{t}$-algebras $R$ given by the isomorphism classes of tuples $(\calF,\al)$ with
\begin{equation}\label{BD_Grass_dfn_equal}
\begin{cases}
\text{$\calF$ a $\calG\otimes_{k\pot{t}}R\pot{z-t}$-torsor on $\Spec(R\pot{z-t})$};\\
\text{$\al\co \calF|_{\Spec(R\rpot{z-t})}\simeq \calF^0|_{\Spec(R\rpot{z-t})}$ a trivialization},
\end{cases}
\end{equation}
where $\calF^0$ denotes the trivial torsor. Here $z$ is an additional formal variable, and the map $k\pot{t}\to R\pot{z-t}$ is the unique $k$-algebra map with the property $t\mapsto z$. By \cite[\S0.3]{Ri19} the functor $\Gr_\calG$ agrees with \cite[Def.~3.3]{Ri16} defined using a spreading of $\calG$ over some curve, and therefore is representable by an ind-projective ind-scheme over $k\pot{t}$ by \cite[Thm.~2.19]{Ri16}. The generic fiber $\Gr_{\calG, F}$ is canonically the affine Grassmannian associated with the reductive group scheme $\calG\otimes_{k\pot{t},t\mapsto z}F\pot{z-t}\simeq G\otimes_FF\pot{z-t}$ (cf.~\cite[Lem.~0.2]{Ri19}), and thus is equivariantly (for the left action of the positive loop group) isomorphic to the usual affine Grassmannian $\Gr_G$ over $F$ formed using the parameter $z-t\in F\pot{z-t}$. Its special fiber is canonically the twisted affine flag variety $\Fl_{G,\bbf}$ for $\calG=\calG_\bbf$ over $\calO_F=k\pot{t}$ in the sense of \cite{PR08}.

Recall we fixed a conjugacy class $\{\mu\}$ of geometric cocharacters in $G$ with reflex field $E/F$. This defines a closed subscheme $\Gr_{G}^{\leq \{\mu\}}\subset \Gr_{G}\otimes_F E$ inside the affine Grassmannian which is a (geometrically irreducible) projective $E$-variety. As in mixed characteristic above (see \cite{Zhu14, Ri16}), the {\em local model \textup{(}or global Schubert variety\textup{)}} $M_{\{\mu\}}=M(G,\calG_\bbf,\{\mu\},t)$ is the scheme theoretic closure of the locally closed subscheme 
\[
\Gr_{G}^{\leq\{\mu\}}\,\hookto\, \Gr_{G}\otimes_F E\,\hookto\, \Gr_{\calG}\otimes_{\calO_F}\calO_E.
\]
By definition, the local model is a reduced flat projective $\calO_E$-scheme, and equipped with an embedding of its special fiber
\[
\overline{M}_{\{\mu\}}:=M_{\{\mu\}}\otimes_{\calO_E}k_E\;\hookto\; \Fl_{G,\bbf}\otimes_kk_E.
\]
Likewise, the admissible locus $\calA(G,\{\mu\})$ is defined as the union of the $(\bbf,\bbf)$-Schubert varieties $S_w\subset \Fl_{G,\bbf}\otimes_k\bar{k}$ where $w$ runs through the elements of the admissible set $\Adm_{\{\mu\}}^\bbf\subset W_\bbf\backslash W/W_\bbf$ inside the double classes in the Iwahori-Weyl group.

\subsubsection{Proof of Theorem \ref{special_fiber_thm} i\textup{)} in equal characteristic} As in the mixed characteristic situation, the inclusion $\calA(G,\{\mu\})\subseteq \overline{M}_{\{\mu\}}\otimes_{k_E}\bar{k}$ is proven as part of \cite[Thm.~5.14]{HRb}. Let $\calV=\Lie \,\calG$ denote the Lie algebra which is a free $\calO_F$-module of rank $\dim_F(G)$. The adjoint representation $\calG\to \GL(\calV)$ induces by functoriality a morphism of ind-projective $\calO_F$-ind-schemes
\[
\ad\co \Gr_{\calG}\;\to\; \Gr_{\GL(\calV)},
\]
where the target is defined as in \eqref{BD_Grass_dfn_equal} using the $\calO_F$-group scheme $\GL(\calV)$. 
Also we let $\calL_{\det}$ be the determinant line bundle on the target, and denote by $\calL:=\ad^*(\calL_{\det})$ its pullback.
Let $\calL_{\sF}$ (resp.~$\calL_{\bar{k}}$) denote the restriction of $\calL$ to the geometric generic (resp.~geometric special) fiber $\Gr_{\calG,\sF}=\Gr_{G,\sF}$ (resp.~$\Gr_{\calG,\bar{k}}=\Fl_{G,\bbf}\otimes_k \bar{k}$). The rest of the argument is the same as in mixed characteristic above using the following lemma.

\begin{lem} \label{line_bundle_equal}
The pullback of  the line bundle $\calL$ to $M_{\{\mu\}}$ is relatively ample over $\calO_E$, and for all $n\geq 1$ one has
\[
\dim_{\bar{F}} \Ga\big(M_{\{\mu\},\sF},\calL_\sF^{\otimes n}\big)\;=\; \dim_{\bar{k}}\Ga\big(\calA(G,\{\mu\}),\calL_{\bar{k}}^{\otimes n}\big).
\]
\end{lem}
\begin{proof} The proof relies on Proposition \ref{normal_prop} and the Coherence Conjecture \cite{Zhu14}, and proceeds in the same steps as in Lemma \ref{line_bundle}. The restriction on the group $G$ is a little milder in equal characteristic due the existence of local models for any, possibly wildly ramified, reductive group.
\end{proof}

\subsection{Proof of Corollary \ref{Invariance_Cor}}\label{Invariance_Cor_Sec}
We treat mixed and equal characteristic by the same argument, so that now the local field $F$ is either a finite extension of $\bbQ_p$ or isomorphic to $k_F\rpot{t}$. Let $(G',\{\mu'\},\calG_{\bbf'})\to (G,\{\mu\},\calG_\bbf)$ be as in Corollary \ref{Invariance_Cor} where the reductive groups $G'\to G$ are defined over $F$ and induce an isomorphism $G_\ad'\simeq G_\ad$ on adjoint groups. As in \cite[Prop.~2.2.2]{KP18}, the map $G'\to G$ induces a map of $\calO_{E'}$-schemes on local models
\begin{equation}\label{adjoint_Loc_Mod}
M_{(G',\{\mu'\},\calG_{\bbf'})}\;\longto\; M_{(G,\{\mu\},\calG_{\bbf})}\otimes_{\calO_{E}}\calO_{E'},
\end{equation}
where $E'/F$ (resp.~$E/F$) denotes the reflex field of $\{\mu'\}$ (resp.~$\{\mu\}$). Note that $E\subset E'$ is naturally a subfield. 

Now on geometric generic fibers \eqref{adjoint_Loc_Mod} is the canonical map $\Gr_{G',\bar{F}}^{\leq \{{\mu'}\}}\to \Gr_{G,\bar{F}}^{\leq \{\mu\}}$ on Schubert varieties which is finite and birational by Proposition \ref{Schubertadprop}. In particular, \eqref{adjoint_Loc_Mod} is birational. To show that \eqref{adjoint_Loc_Mod} is finite, we observe that this map is proper (because source and target are proper), and hence it is enough, by \cite[0A4X]{StaPro}, to show that \eqref{adjoint_Loc_Mod} is quasi-finite. As we already know that \eqref{adjoint_Loc_Mod} is (quasi-)finite in generic fibers, it remains to show that it is quasi-finite on (reduced geometric) special fibers. By \cite[Thm.~5.14]{HRb}, the map \eqref{adjoint_Loc_Mod} identifies on reduced geometric special fibers with the canonical map $\calA(G',\{\mu'\})\to \calA(G,\{\mu\})$. 
Applying Proposition \ref{Schubertadprop} again, we see that the latter map is finite. 
This shows that \eqref{adjoint_Loc_Mod} is birational and finite.

For universal homeomorphism, we have to show that map \eqref{adjoint_Loc_Mod} is integral, universally injective and surjective, cf.~\cite[04DC]{StaPro}. 
Being finite this map is integral.
Universally injective and surjective can be checked on geometric points over the fibers of $\Spec(\calO_{E'})$ where it again follows from Proposition \ref{Schubertadprop}.

Now assume the normality of $\Gr_{G,\bar{F}}^{\leq \{\mu\}}$ and all maximal $(\bbf^\flat,\bbf^\flat)$-Schubert varieties inside $\calA(G,\{\mu\})$. Then Theorem \ref{special_fiber_thm} i) applies to show that the special $M_{(G,\{\mu\},\calG_{\bbf}),k_{E'}}$ is reduced (because it is geometrically reduced). Since its generic fiber is normal, we can apply \cite[Prop.~9.2]{PZ13} to prove the normality of $M_{(G,\{\mu\},\calG_{\bbf})}\otimes_{\calO_{E}}\calO_{E'}$. As the map \eqref{adjoint_Loc_Mod} is finite and birational, it must be an isomorphism.


\section{Cohen-Macaulayness of local models}

\subsection{Recollections on Frobenius splittings}
Let $X$ be a scheme in characteristic $p>0$, and denote by $F=F_X\co X\to X$ its absolute $p$-th power Frobenius. The scheme $X$ is called {\em Frobenius split} if the map $\calO_X\to F_*\calO_X$ splits as a map of $\calO_X$-modules. In this case a splitting map $\varphi\co F_*\calO_X\to \calO_X$ is called a Frobenius splitting. Also recall the following notions.

\begin{dfn} \label{splitting} Let $X$ be Frobenius split. 
\begin{enumerate}
\item[i)] We say $X$ splits {\em compatibly} with some closed subscheme $Z=V(I)\subset X$ if there exists a splitting $\varphi\co F_*\calO_X\to \calO_X$ such that $\varphi(F_*I)\subset I$.
\item[ii)] We say $X$ splits {\em relative} to some effective Cartier divisor $D\subset X$ if the composition $\calO_X\to F_*\calO_X\hookto F_*\calO_X(D)$ splits as a map of $\calO_X$-modules.
\end{enumerate}
\end{dfn}

\begin{lem}\label{nice_lem}
Let $X$ be Frobenius split, and let $D\subset X$ be an effective Cartier divisor. If $X$ splits compatibly with $D$, then $X$ splits relative to $(p-1)\cdot D$. In this case, $X$ splits relative to $D$.
\end{lem}
\begin{proof}
Let $\varphi\co F_*\calO_X\to \calO_X$ be a splitting compatible with $D=V(\calO_X(-D))$, i.e., $\varphi$ restricts to a splitting $F_*\calO_X(-D)\to \calO_X(-D)$. By tensoring with $\calO_X(D)$ we obtain a splitting 
\[
F_*\calO_X\big((p-1)\cdot D\big)=F_*\big(\calO_X(-D)\otimes_{\calO_X} F^*\calO_X(D)\big) =F_*\calO_X(-D)\otimes_{\calO_X}\calO_X(D)\to \calO_X,
\]
i.e., $X$ splits relative to $(p-1)\cdot D$. The last assertion is immediate from the factorization $\calO_X\to F_*\calO_X(D)\hookto F_*\calO_X((p-1)\cdot D)$.
\end{proof}

\begin{rmk} If $X$ is a smooth variety over an algebraically closed field, then the converse to Lemma \ref{nice_lem} holds. Namely, $X$ splits compatibly with $D$ if and only if $X$ splits relative to $(p-1)\cdot D$. This is stated in \cite[Thm.\,1.4.10]{BK05}, but we do not need this sharper result.
\end{rmk}


\begin{lem} \label{split_lem}
Let $X$ be Frobenius split compatibly with closed subschemes $Z_1, Z_2\subset X$. Then each $Z_i$, $i=1,2$ is Frobenius split compatibly with $Z_1\cap Z_2$.
\end{lem}
\begin{proof} For $i=1,2$, let $Z_i=V(I_i)$ and let $\varphi$ be a splitting with $\varphi(F_*I_i)\subset I_i$. This automatically induces splittings on each $Z_i$, and it is elementary to see that $\varphi(F_*(I_1+I_2))\subset I_1+I_2$. Since $V(I_1+I_2)=Z_1\cap Z_2$, the lemma follows. 
\end{proof}

\begin{prop} \label{CM_split}
Let $X$ be Frobenius split and locally of finite type over a field \textup{(}or a Dedekind domain\textup{)}, and let $D\subset X$ be an effective Cartier divisor. If $X$ splits relative to $D$ and $X\backslash D$ is Cohen-Macaulay, then $X$ is Cohen-Macaulay. 
\end{prop}
\begin{proof} This is \cite[Ex.~5.4]{BS13}, and the following proof was communicated to us by K.~Schwede\footnote{Of course, any insufficiencies in the presentation are entirely due to the authors.}.
We have to show that all local rings $\calO_{X,x}$ for $x\in D$ are Cohen-Macaulay. Without loss of generality we may assume that $X=\Spec(R)$ where $(R,\frakm)$ is a Noetherian local ring and $D=V(f)$ for some non-zero divisor $f\in \frakm$. By \cite[0AVZ]{StaPro}, we have to show that the local cohomology $H^i_\frakm(R)$ vanishes for $i=0,\ldots,d-1$, $d:=\dim(R)$. Our finiteness assumptions on $X$ imply that $R$ admits a dualizing complex (cf.~ \cite[0BFR]{StaPro}) so that Lemma \ref{torsion_local_coho} below applies. Hence, there exists an $N>\!\!>0$ with $f^N\cdot H^i_{\frakm}(R)=0$ for all $i=0,\ldots, d-1$. By \cite[Lem.~5.2.3]{BS13}, there exists for any $e\in \bbZ_{\geq 1}$ a splitting of the composition
\begin{equation}\label{higher_split}
R\;\longto\; F_*^eR\;\overset{F_*^e\!f^{q_e}}{\longto} F_*^eR,
\end{equation}
where $q_e:=1+p+\ldots+p^{e-1}$, i.e., $\Spec(R)$ is $F_*^e$-split relative to $q_e\cdot V(f)$. Now choose $e>\!\!>0$ such that $q_e\geq N$, i.e., $f^{q_e}$ kills the local cohomologies as above. Finally, consider the sequence of $R$-modules 
\[
H^i_\frakm(R)\;\longto\; H^i_\frakm(F_*^eR)\;\overset{F_*^e\!f^{q_e}\,}{\longto}\; H^i_\frakm(F_*^eR).
\]
Using $H^i_\frakm(F_*^eR)=F_*^eH^i_\frakm(R)$ ($F_*$ is exact) we see that $F_*^e\!f^{q_e}$ induces the zero map on $H^i_\frakm(F_*^eR)$ for all $i=0,\ldots,d-1$. By virtue of the splitting \eqref{higher_split} this means that the identity map on $H^i_\frakm(R)$ factors through $0$, or equivalently $H^i_\frakm(R)=0$. The lemma follows.
\end{proof}

\begin{lem}\label{torsion_local_coho}
Let $(R,\frakm,k)$ be a Noetherian local ring of dimension $d$ which admits a dualizing complex in the sense of \cite[0A7B]{StaPro}. Let $f\in R$ be a non-zero divisor such that the localization $R_f$ is Cohen-Macaulay. For any finite $R$-module $M$ whose localization $M_f$ is a projective $R_f$-module, there exists an integer $N>\!\!>0$ such that the local cohomology vanishes
\[
f^N\cdot H^i_{\frakm}(M) \;=\; 0
\]
for all $i=0,\ldots, d-1$.
\end{lem}
\begin{proof}
By the local duality theorem \cite[0AAK]{StaPro}, we have 
\[
\Ext^{-i}_R(M,\om_R^\bullet)^\wedge =\Hom_R(H^i_\frakm(M),E) = \Hom_{\hat{R}}(H^i_\frakm(M),E),
\]
where $(\cdot)^\wedge$ denotes the $\frakm$-adic completion, $\om_R^\bullet$ the normalized dualizing complex (cf.~\cite[0A7B]{StaPro}) and $E$ the injective hull of $k$. As $R_f$ is Cohen-Macaulay, the localized complex $(\om_R^\bullet)_f=\om_{R_f}[d]$ is concentrated in degree $d$ (cf.~\cite[0A86, 0AWS]{StaPro}), so that 
\[
\Ext^{-i}_R(M,\om_R^\bullet)_f\;=\;\Ext^{d-i}_{R_f}(M_f,\om_{R_f})\;=\;0
\] 
for $i=0,\ldots,d-1$ where we have used that $M_f$ is projective over $R_f$ for the last equation. Recall that $\om_R^\bullet$ is a cohomologically bounded complex whose cohomology groups are finite $R$-modules (cf.\,\cite[p.257]{Ha66} or \cite[0A7B]{StaPro}). Since $M$ is finite, each $R$-module $\Ext^{-i}_R(M,\om_R^\bullet)$ is finite as well, and hence there exists a uniform $N>\!\!>0$ such that $f^N$ kills each module $\Ext^{-i}_R(M,\om_R^\bullet) \otimes_R \hat{R} = \Ext^{-i}_R(M,\om_R^\bullet)^\wedge$. We conclude that $f^N$ kills each $H^i_\frakm(M)$ by Matlis duality, cf.~\cite[08Z9]{StaPro}, bearing in mind that $H^i_\frakm(M)$ satisfies the DCC because $M$ is finite and $(R,\frakm)$ is Noetherian local, cf.\,\cite[Prop.\,3.5.4]{BrHe}.
\end{proof}

\subsection{Proof of Theorem \ref{special_fiber_thm} ii\textup{)}}\label{CM_proof}
We begin with a general lemma.

\begin{lem}\label{elementary_lem}
Let $X$ be a flat scheme of finite type over a discrete valuation ring. 
\begin{enumerate}
\item[i)] Assume that the generic fiber $X_\eta$ is normal and the special fiber $X_s$ is reduced. Then $X$ is normal. 
\item[ii)] Assume that the generic fiber $X_\eta$ is Cohen-Macaulay. Then $X$ is Cohen-Macaulay if and only if its special fiber $X_s$ is Cohen-Macaulay if and only if its geometric special fiber $X_{\bar{s}}$ is Cohen-Macaulay.   
\end{enumerate}
\end{lem}
\begin{proof} Part i) is \cite[Prop.~9.2]{PZ13}, and ii) is immediate from \cite[0C6G, 045P]{StaPro}. 
\end{proof}

Now let $F$ be a non-archimedean local field (either of mixed or equal characteristic), and fix a triple $(G,\{\mu\},\calG_\bbf)$ as in \S\ref{main_result_sec} where $G$ is defined over $F$. Let $M:=M_{(G,\{\mu\},\calG_\bbf)}$ be the associated local model over $\calO_E$, where $E$ is the reflex field. As the geometric generic fiber $M_{\bar F}$ is, by definition, a Schubert variety inside an affine Grassmannian which we assume to be normal, it is Cohen-Macaulay by Proposition \ref{normal_prop} i). By \cite[0380, 045V]{StaPro} the generic fiber $M_E$ is normal and Cohen-Macaulay as well. In view of Theorem \ref{special_fiber_thm} i) and Lemma \ref{elementary_lem} this implies that $M$ is normal. Here we are assuming that each maximal Schubert variety inside the admissible locus $\calA(G,\{\mu\})$ is normal. It remains to show that if $p>2$ then $M$ is also Cohen-Macaulay. By Lemma \ref{elementary_lem} ii) this is equivalent to the Cohen-Macaulayness of the geometric special fiber $M_{\bar{k}}=\calA(G,\{\mu\})$, cf.~Theorem \ref{special_fiber_thm} i). As the combinatorics of Iwahori-Weyl groups are the same in mixed and equal characteristic (cf.~\cite[Lem.~4.11]{HRb} for a precise statement), we may and do assume that $F\simeq k\rpot{t}$ is of equal characteristic, i.e., $M$ is a scheme in characteristic $p>0$. For this, we must remark that the group $G^\flat$ in (\ref{Gflat}) arising from the mixed characteristic situation is tamely ramified so that the tame ramification hypotheses on $G^\flat_\scon$ used in Theorem \ref{special_fiber_thm} is satisfied; this holds by the description of the maximal torus $\uT_1 \subset \uG_1$ given in \cite[Ex.\,4.14]{HRb}. Also we may and do assume that $k=\bar{k}$ is algebraically closed because the formation of local models commutes with unramified base change. Further, by \eqref{decom_adm_locus} the admissible locus takes the form
\begin{equation}\label{product_decom}
M_k\simeq \calA(G,\{\mu\})\simeq \calA(G,\{\mu\})^o\simeq \prod_{j\in J}\prod_{i=1}^{l_j}\calA(G_j,\{\mu_{i,1}^{(j)}\}+\ldots \{\mu_{i,m_j}^{(j)}\})^o,
\end{equation}
where $G_\ad=\prod_{j\in J}\Res_{F_j/F}(G_j)$ for absolutely simple $F_j$-groups $G_j$ and $m_j:=[F_j:F_j^{\on{ur}}]$, $l_j:=[F_j^{\on{ur}}:F]$ for the maximal unramified subextension $F_j/F_j^{\on{ur}}/F$. As products of locally Noetherian flat Cohen-Macaulay schemes are Cohen-Macaulay (cf.~\cite[0C0W, 045J]{StaPro}), we see that it is enough to proof the Cohen-Macaulayness of each $\calA(G_j,\{\mu_{i,1}^{(j)}\}+\ldots \{\mu_{i,m_j}^{(j)}\})^o$, i.e., we may and do assume that $G$ is absolutely almost simple. Also note that under \eqref{product_decom} the Schubert varieties in each factor are still normal by Proposition \ref{normal_prop} ii) so that our normality assumption still holds for the Schubert varieties in each absolutely almost simple factor.   

Summarizing the discussion, we have an {\em equal characteristic} local model $M$ attached with some absolutely almost simple group (so that Lemma \ref{key_lem_split} below is available) which satisfies the normality assumptions of Theorem \ref{special_fiber_thm}. We know that $M$ is a flat projective scheme over $\calO=\calO_E$ which is normal. Its generic fiber $M_E$ is Cohen-Macaulay, and its special fiber $M_k$ (=the admissible locus) is an effective Cartier divisor on $M$. We aim to show that $M$ is, as a whole, Cohen-Macaulay. 

The key to the proof is now the following lemma which is a direct consequence of \cite[Thm.~6.5]{Zhu14} (also this is the key step in the proof of the Coherence Conjecture).

\begin{lem}\label{key_lem_split}
Let $p>2$. Then the local model $M$ is Frobenius split compatibly with its special fiber $M_k\subset M$ viewed as a closed subscheme.
\end{lem}
\begin{proof} In \cite[Thm.~6.5]{Zhu14}, a $\calO$-scheme $X$ together with a closed immersion $M\subset X$ is constructed such that $X$ is Frobenius split compatibly with both $M$ and its special fiber $X_k$. Hence, Lemma \ref{split_lem} implies that $M$ is Frobenius split compatibly with $M\cap X_k=M_k$.
\end{proof}

\begin{cor}
For $p>2$, the local model $M$ splits relative to its special fiber $M_k\subset M$ viewed as an effective Cartier divisor.
\end{cor}
\begin{proof}
This is follows from Lemmas \ref{key_lem_split} and \ref{nice_lem}.

\end{proof}

As we already know that $M\backslash M_k=M_E$ is Cohen-Macaulay, we can now apply Proposition \ref{CM_split} to conclude that $M$ (and hence $M_k=\calA(G,\{\mu\})$) is Cohen-Macaulay. It remains to identify the dualizing sheaf on the local model.

\begin{lem}
Let $M=M_{(G,\{\mu\}, \calG)}$ be a local model in either mixed or equal characteristic defined over the discrete valuation ring $\calO=\calO_E$. If $M$ is normal and Cohen-Macaulay, then the dualizing sheaf is given by 
\[
\om_M\;=\; (\Om_{M/\calO}^d)^{*,*},
\]
where $d=\dim(M_E)$ is the dimension of the generic fiber.
\end{lem}  

\begin{proof} Both sheaves  $(\Om_{M/\calO}^d)^{*,*}$ and $\om_M$ are reflexive: for the first this is clear, and for the second this is \cite[Lem.~3.7.5]{Kov} using the normality of $M$. Let $U:=(M)^{\on{sm}}$ be the locus which is smooth over $\calO$. It follows from \cite[Thm.~6.12]{HRa} (and \cite[Thm.~5.14]{HRb} for Weil restricted groups in mixed characteristic) that the complement $M\backslash U$ has codimension $\geq 2$. Also, by \cite[0EA0]{StaPro}, there is a map $\Om^d_{M/\calO}\to\om_M$ which is an isomorphism restricted to $U$. Thus, we get an isomorphism
\[
\big(\Om^d_{M/\calO}\big)^{*,*}|_U\simeq \Om_{U/\calO}^d\simeq \om_U=\om_M|_U,
\]
which by the normality of $M$ and the reflexivity of both sheaves $(\Om_{M/\calO}^d)^{*,*}$, $\om_M$ extends to all of $M$, cf.~\cite[Prop.~1.6]{Ha80} (see also \cite[0EBJ]{StaPro}).
\end{proof}

\end{document}